\documentclass[smallextended,referee,envcountsect]{svjour3}
\usepackage [latin1]{inputenc}
\usepackage{amsmath,amssymb}
\usepackage{marvosym,mathtools}
\usepackage[numbers,sort&compress]{natbib}
\usepackage[colorlinks,linkcolor=blue,urlcolor=blue,citecolor=blue]{hyperref}
\usepackage{graphicx,subfig}
\usepackage{float}
\usepackage{epstopdf}
\smartqed
\usepackage{graphicx}
\journalname{}

\usepackage{fancyhdr}
\usepackage{ntheorem}
\theoremheaderfont{\bfseries\upshape}
\theorembodyfont{\upshape}
\renewtheorem{remark}{\it Remark}[section]
\renewtheorem{example}{Example}[section]
\newtheorem{assumption}{Assumption}[section]

\begin{document}

\title{Primal-dual dynamics  featuring Hessian-driven damping and variable mass  for convex optimization problems}

\author{Xiangkai Sun$^{1}$\and Feng Guo$^{1}$\and Liang He$^{1}$\and Xiaole Guo$^{1}$}

\institute{
          \\ Xiangkai Sun  \at {\small sunxk@ctbu.edu.cn} \\
          \\ Feng Guo \at{\small 993551272@qq.com} \\
          \\Liang He\at{\small liangheee@126.com}  \\
           \\Xiaole Guo (\Letter) \at{\small xlguocqu1@163.com}  \\
                     \\
               $^{1}$Chongqing Key Laboratory of  Statistical Intelligent Computing and Monitoring, College of Mathematics and Statistics,
 Chongqing Technology and Business University,
Chongqing 400067, China. }

\date{Received: date / Accepted: date}

\maketitle

\begin{abstract}
This paper deals with a new Tikhonov regularized primal-dual dynamical system with variable mass and Hessian-driven damping for solving a convex optimization problem with linear equality constraints. The system features several  time-dependent parameters: variable mass, slow viscous damping, extrapolation, and temporal scaling. By employing the Lyapunov analysis approach, we obtain the strong convergence of the trajectory generated by the proposed system to the minimal norm  solution of the optimization problem, as well as convergence rate results for the primal-dual gap, the objective residual, and the feasibility violation. We also show that the convergence rates of the primal-dual gap, the objective residual, and the feasibility violation can be improved by appropriately adjusting these parameters. Further, we conduct numerical experiments to demonstrate the effectiveness of the theoretical results.
\end{abstract}

\section{Introduction}

Let $\mathcal{X}$ and $\mathcal{Y}$ be two real Hilbert spaces. Let $f:\mathcal{X}\rightarrow\mathbb{R}$ be a continuously differentiable convex function, $A:\mathcal{X}\rightarrow\mathcal{Y}$ be a linear operator and $b\in\mathcal{Y}$.   The convex optimization problem with linear equality constraints is defined as
\begin{eqnarray}\label{constrained}
\left\{ \begin{array}{ll}
&\mathop{\mbox{min}}\limits_{x\in\mathcal{X}}~~{f(x)}\\
&\mbox{s.t.}~~Ax=b.
\end{array}
\right.
\end{eqnarray}

In recent years, as one of the powerful frameworks for   solving   problem  (\ref{constrained}), the inertial primal-dual dynamics approach  has attracted great interest of many scholars. In order to solve   problem (\ref{constrained}), Zeng et al. \cite{Zeng}   propose the following   primal-dual inertial dynamical system with asymptotically vanishing damping:
\begin{equation}\label{re4}
\left\{ \begin{split}
&\ddot{x}(t)+\frac{\alpha}{t}\dot{x}(t)+\nabla_{x} {L} (x(t),\lambda(t)+\theta t\dot{\lambda}(t))=0,\\
& \ddot{\lambda}(t)+\frac{\alpha}{t}\dot{\lambda}(t)-\nabla_{\lambda} {L} (x(t)+\theta t\dot{x}(t),\lambda(t))=0, \end{split}
\right.
\end{equation}
where $\alpha\geq3$, $\theta=\frac{1}{2}$,  and $ {L}: \mathcal{X}\times \mathcal{Y}\rightarrow \mathbb{R}$ is the augmented Lagrangian function of problem (\ref{constrained})  defined as $L(x ,\lambda )=f(x)+\langle\lambda,Ax-b\rangle+\frac{1}{2}\|Ax-b\|^2.$ They show that the fast convergence rates for the primal-dual gap, the feasibility violation  and the velocity vector along the trajectory of  system (\ref{re4}) are $\mathcal{O}\left(\frac{1}{t^2}\right)$, $\mathcal{O}\left(\frac{1}{t }\right)$ and $\mathcal{O}\left(\frac{1}{t  }\right)$, respectively.
It is well-known that the time scaling technique is an efficient way to further improve the rates of convergence in   problem (\ref{constrained}). Then,
 Hulett and Nguyen \cite{Hulett2023T} introduce the second-order primal-dual inertial dynamical system with generalized time scaling and asymptotically vanishing damping:
\begin{equation*}\label{hulett23}
\left\{ \begin{split}
&\ddot{x}(t)+\frac{\alpha}{t}\dot{x}(t)+\beta(t)\nabla_{x} {L}^{\rho} (x(t),\lambda(t)+\theta t\dot{\lambda}(t))=0,\\
 &\ddot{\lambda}(t)+\frac{\alpha}{t}\dot{\lambda}(t)-\beta(t)\nabla_{\lambda} {L}^{\rho} (x(t)+\theta t\dot{x}(t),\lambda(t))=0,
 \end{split}
\right.
\end{equation*}
where $\alpha\geq0$, $\theta>0$,  $\beta: [t_0,+\infty)\rightarrow \mathbb{R}$ is time scaling function and   $L^\rho(x(t) ,\lambda(t) )=f(x(t))+\langle\lambda(t),Ax(t)-b\rangle+\frac{\rho}{2}\|Ax(t)-b\|^2$ with the
penalty parameter $\rho>0$. They  derive faster convergence
rates for the primal-dual gap, the feasibility violation, and the objective function value along the generated trajectories,   which can be regarded as an extension of the results obtained in \cite{Zeng,botjde,heaa23}. For problem (\ref{constrained}) with a separable structure, He et al. \cite{He2021C} and Attouch et al. \cite{ajota} also explore  second-order primal-dual dynamical systems with time-dependent
damping coefficients   and obtain  some results similar to those in \cite{Zeng,botjde,heaa23}.

Recently, many scholars have studied the  ``second-order primal''+ ``first-order
dual'' dynamical system for solving problem (\ref{constrained}) since first-order ordinary differential equations (ODEs) are generally   easier
to solve than second-order ODEs  from the perspective of numerical computation. More precisely, by using
the heavy-ball accelerated constant damping $\alpha> 0$,  He et al. \cite{He2022S} introduce  a ``second-order primal''+ ``first-order
dual'' dynamical system with constant viscous damping and time scaling:
\begin{equation}\label{re6}
\left\{ \begin{split}
&\ddot{x}(t)+\alpha\dot{x}(t)+\beta(t)\nabla_{x} {L}^\rho(x(t),\lambda(t))=0,\\
&\dot{\lambda}(t)-\beta(t)\nabla_{\lambda} {L}^\rho(x(t)+\theta \dot{x}(t),\lambda(t))=0. \end{split}
\right.
\end{equation}
It is worth noting that the  system (\ref{re6}) involves the inertial
term only for the primal variable.
They demonstrate that, in the case where
$f$ is a convex function, the system (\ref{re6})  enjoys a convergence rate of
$\mathcal{O}\left(\frac{1}{\beta(t) }\right)$, as $t\rightarrow +\infty$.
For problem (\ref{constrained}) where  the objective function $f$ is $\mu$-strongly convex, He et al. \cite{He2026} introduce  a ``second-order primal''+ ``first-order dual'' dynamical system defined as
\begin{equation*}
\left\{ \begin{split}
&\ddot{x}(t)+2\sqrt{\mu}\dot{x}(t)+\alpha\nabla_{x} {L}^\rho(x(t),\lambda(t))=0,\\
&\dot{\lambda}(t)-\beta(t)\nabla_{\lambda} {L}^\rho\left(x(t)+\frac{1}{\sqrt{\mu}} \dot{x}(t),\lambda(t)\right)=0,\end{split}
\right.
\end{equation*}
where $\mu$ is the strong convex coefficients of $f$ and $\alpha\geq1.$ Then, they establish an $\mathcal{O}\left(\frac{1}{\beta(t) }\right)$ convergence rate and demonstrate that it can achieve an optimal rate of $\mathcal{O}\left( e^{-\sqrt{\mu}t}\right)$. More results
on the convergence rates of  ``second-order primal''+ ``first-order dual'' dynamical systems for   problem (\ref{constrained}) can be found in \cite{He2022,jiang}.

Very recently, to ensure the trajectory converges strongly to the minimal norm solution of problem (\ref{constrained}), instead of weakly to an arbitrary minimizer,
 Zhu et al. \cite{2026zhu} introduce  a Tikhonov regularized ``second-order primal''+ ``first-order dual'' dynamical system with asymptotically vanishing
damping:
   \begin{equation}\label{zhu26}
\left\{ \begin{split}
&\ddot{x}(t)+\frac{\alpha}{t}\dot{x}(t)+ \nabla_{x} {L}^{\rho} (x(t),\lambda(t))+\epsilon(t)x(t)=0,\\
 &\dot{\lambda}(t)-t\nabla_{\lambda} {L}^{\rho} \left(x(t)+ \frac{t}{\alpha-1} \dot{x}(t),\lambda(t)\right)=0.\end{split}
\right.
\end{equation}
Here  $\epsilon: [t_0, +\infty) \rightarrow[0, +\infty)$ is the Tikhonov
regularization parameter which is a $\mathcal{C}^1$ nonincreasing function satisfying
$\lim_{t\rightarrow+\infty} \epsilon(t) = 0$. Under some mild assumptions on $ \epsilon(t)$, they not only derive the $\mathcal{O}\left(\frac{1}{t^2}\right)$
  convergence rates of the primal-dual gap, the objective
residual, and the feasibility violation along the generated trajectory, but also
  prove the strong convergence of the primal trajectory of
  system (\ref{zhu26})  to the minimal
norm solution of    problem (\ref{constrained}).
Subsequently, Li et al. \cite{lihl}  improve the convergence rates of the work in \cite{2026zhu}  by employing  a Tikhonov regularized ``second-order primal''+ ``first-order dual'' dynamical system with general viscous damping and time scaling functions. Further, following the approach used in \cite{2026zhu,zhujcam}, Sun et al. \cite{sunjota} also investigate a Tikhonov regularized ``second-order primal''+ ``first-order dual'' dynamical system for  problem (\ref{constrained}) with a separable structure.

It is worth noting that in \cite{2026zhu,lihl,zhujcam,sunjota},  the strong convergence of the trajectory $x(t)$  towards  the minimal norm solution $x^*$ of
problem (\ref{constrained}) is only ensured under a strong assumption that, for sufficiently large $t$,
$x(t)$ either stays in the open ball $\mathbb{B}(0, \|x^*\|)$, or in its complement.
To address this, Battahi et al. \cite{battahi25} introduced the following Tikhonov regularized ``second-order primal''+ ``first-order dual'' dynamical system for problem (\ref{constrained}):
 \begin{equation}\label{batt25}
\left\{ \begin{split}
&\ddot{x}(t)+ \alpha \dot{x}(t)+ t^p\nabla_{x} \mathcal{L}_t (x(t),\lambda(t)) =0,\\
 &\dot{\lambda}(t)-t^p\nabla_{\lambda} \mathcal{L}_t \left(x(t)+ \frac{1}{\tau} \dot{x}(t),\lambda(t)\right)=0,\end{split}
\right.
\end{equation}
where $\alpha > 0$ is a damping parameter, $t^p$ is  the temporal scaling with $r>0$, $\frac{1}{\tau}$ is the extrapolation parameter, and $ \mathcal{L}_t: \mathcal{X} \times \mathcal{Y} \to \mathbb{R}$ is the augmented
Lagrangian saddle function defined as (\ref{asd}). By appropriately adjusting
these parameters, they
show   that    the fast convergence rates of the primal-dual gap, the feasibility violation, and the objective  residual along the trajectory of  system (\ref{batt25}) are  $\mathcal{O}\left(\frac{1}{t^r}\right)$. They also  prove that the primal trajectory strongly converges to the minimal norm solution without any strong assumptions. The system (\ref{batt25})  with more general
damping, i.e., $t^p=\beta(t)$, was also addressed by  Battahi et al. in \cite{battahi26}. Recently, to solve problem (\ref{constrained}),  Zhu et al. \cite{Zhu2024S} introduce a more general  Tikhonov regularized   dynamical system:
\begin{equation}\label{zhu}
\left\{ \begin{split}
&\ddot{x}(t)+\frac{\alpha}{t^q}\dot{x}(t)+t^s \nabla_{x}\mathcal{L}_t(x(t),\lambda(t))  =0,\\
&\dot{\lambda}(t)-t^{q+s} \nabla_{\lambda}\mathcal{L}_t(x(t)+\theta t^q \dot{x}(t),\lambda(t))  =0,\end{split}
\right.
\end{equation}
where $0\leq q<1$, $0<p<1$, $\alpha>0$, $s>0$ and $\theta>0$. Note that the system (\ref{zhu}) is
more general than  system (\ref{batt25}) even in the specific case of $q = 0$, since the  parameter $s$ does not need to be equal to the  parameter $p$.
By setting the involved parameters, they derive the convergence rate results for the primal-dual gap, the objective residual, and the feasibility violation, as well as the strong convergence of the solution trajectory of  (\ref{zhu}) to the minimum-norm solution of problem (\ref{constrained}).

On the other hand, as we know, inertial dynamical systems  incorporating Hessian-driven
damping exhibit extensive applicability in the fields of optimization and mechanics.  It should be pointed out that for the unconstrained optimization problem $\min_{x\in \mathcal{X}}f(x)$, there has been a large number of works  devoted to  dynamical systems  with Hessian-driven damping from several different perspectives. See, for example,  \cite{Att14siam,Attouch2016F,Bot2021T,att22mp,Attouch2023A,jems23,bagy23,cs24k,Zhong2024F,L2024S,he26coap}. However, in contrast to unconstrained optimization problem $\min_{x\in \mathcal{X}}f(x)$,  there
exist only few papers devoting to the investigation of  inertial dynamical systems with Hessian-driven damping for solving problem (\ref{constrained}).
 More precisely, He et al. \cite{He2023C} proposed the following ``second-order primal''+ ``first-order dual''  dynamical system with general Hessian-driven damping:
\begin{equation}\label{re9}
\left\{ \begin{split}
&\ddot{x}(t)+\frac{\alpha}{t}\dot{x}(t)+\gamma(t)\frac{d}{dt} \nabla_x\mathcal{L}(x(t),\lambda(t)) +\beta(t)\nabla_{x}\mathcal{L}(x(t),\lambda(t))=0,\\
&\dot{\lambda}(t)-\eta(t)\nabla_{\lambda}\mathcal{L}\left(x(t)+\frac{t}{\alpha-1} \dot{x}(t),\lambda(t)\right)=0,\end{split}
\right.
\end{equation}
where $\alpha>1$, $\gamma:\left[t_0,+\infty \right) \to \left[0,+\infty \right) $ is the  Hessian-driven  damping function,  $ \beta,$ $\eta:\left[t_0,+\infty \right) \to \left[0,+\infty \right) $ are the time scaling functions. They show that the fast convergence rates of the Lagrangian residual, the objective residual and the feasibility violation along the trajectory  of   system (\ref{re9}) are $\mathcal{O}\left(\frac{1}{t \eta(t)}\right)$. Csetnek and L\'{a}szl\'{o} \cite{cs24} also consider a Tikhonov regularized second-order primal-dual dynamical system with Hessian-driven  damping and obtain strong convergence of the trajectories to the minimal norm primal-dual solution, as well as fast convergence rates of the   feasibility measure, velocities and
objective function residual. Sun et al. \cite{sunopt} also investigate a second-order primal-dual dynamical system
with Hessian-driven damping and Tikhonov regularization terms in
connection with a convex-concave bilinear saddle point problem.

Motivated by the works reported in \cite{Zhu2024S,He2023C, sunopt},  for solving problem $(\ref{constrained})$, we introduce the  following Tikhonov regularized  primal-dual dynamical system with variable mass, slowly viscous damping,   Hessian-driven damping and time scaling,
\begin{equation}\label{dyn1}
\left\{ \begin{split}
&m(t)\ddot{x}(t)+\frac{\alpha}{t^q}\dot{x}(t)+\gamma\frac{d}{dt} \nabla_{x}\mathcal{L}_t\left(x(t),\lambda(t)\right) +t^s\nabla_{x}\mathcal{L}_t\left(x(t),\lambda(t)\right)=0,\\
 &\dot{\lambda}(t)-(\alpha-1)(t^{q+s}-\gamma q t^{q-1})\nabla_\lambda \mathcal{L}_t\left(x(t)+\theta (t)\dot{x}(t),\lambda(t)\right)=0,\end{split}
\right.
\end{equation}
where $\alpha>1$, $0<q<1$, $\gamma>0$ and $ s>0 $, $m:[t_0,+\infty)\rightarrow(0,+\infty)$ is a differentiable and monotonically non-increasing function, $ \mathcal{L}_t : \mathcal{X} \times \mathcal{Y} \to \mathbb{R}$ is the augmented Lagrangian saddle function (see  (\ref{asd})  for details), $\frac{\alpha}{t^q}  $  is the slowly viscous damping parameter, $\gamma$ is the constant Hessian-driven damping parameter, $ t^s $ is the time scaling parameter, and  $ \theta(t) $ is the extrapolation parameter defined as
\begin{equation}\label{theta}
  \theta(t)\coloneqq\frac{m(t)t^{2q+s}+\gamma t^q-2m(t)\gamma qt^{2q-1}-\gamma\dot{m}(t)t^{2q}}{(\alpha-1)(t^{q+s}-\gamma q t^{q-1})} .
\end{equation}
The contributions of this paper can be more specially stated as follows:
\begin{enumerate}
\item[{\rm (i)}] We propose a Tikhonov regularized ``second-order primal''+ ``first-order dual'' dynamical system (\ref{dyn1}), which incorporates  variable mass,  slowly viscous  damping  and Hessian-driven damping  terms, for solving the   linearly constrained convex optimization problem (\ref{constrained}).  Our dynamical system (\ref{dyn1}) can be regarded as a generalization of the dynamical system  with variable mass from  \cite{csim2024}  for solving  unconstrained optimization problem $\min_{x\in \mathcal{X}}f(x)$.  It is worth noting that the variable mass term plays a   crucial role in improving the convergence rates of system  (\ref{dyn1}). Meanwhile, the   slowly viscous damping term $\frac{\alpha}{t^q}$ has the role of   achieving  the strong convergence of the
solution trajectory to the minimal-norm
solution as observed in the context of \cite{battahi25,battahi26, Zhu2024S}, rather than the strong convergence in the inferior limit sense. Further,  the role of the Hessian-driven damping  term is to suppress the oscillations  of trajectories  associated with system (\ref{dyn1}).
\item[{\rm (ii)}] Under appropriate settings on the underlying parameters, we obtain the strong convergence of the trajectory to the minimal norm  solution of  problem $(\ref{constrained})$, as well as fast convergence rates of the primal-dual gap, the objective residual, and the feasibility violation.   Compared with the results   on strong convergence in the inferior limit sense obtained in  \cite{2026zhu,lihl,zhujcam,sunjota,Bot2021T,cs24k}, the  strong convergence  result obtained in this paper is   in the limit sense  and does not require  the strong assumptions  that $x(t)$ either stays in the open ball $\mathbb{B}(0, \|x^*\|)$, or in its complement.
\item[{\rm (iii)}] We perform  two numerical examples to  demonstrate  the efficiency of   system  (\ref{dyn1})   in terms of  the objective residual and the feasibility violation.  In the first numerical example, we demonstrate that the system (\ref{dyn1})  incorporating a variable mass term can preserve  and even improve  the results regarding the fast convergence rate. In the second numerical example, we show that  the system (\ref{dyn1}) with a Hessian-driven damping term
    can eliminate  possible oscillations in the dynamical behaviour of the trajectories.
\end{enumerate}

The rest of this paper is organized as follows. In Section 2,  we recall some basic notations and present some preliminary results. In Section 3, we investigate the convergence properties of the primal-dual gap, the objective function value and the feasibility violation, and the strong convergence of the primal-dual trajectory generated by system (\ref{dyn1}).  In Section 4, we give some numerical experiments to illustrate our theoretical findings.

\section{Preliminaries}

Let $h:\mathcal{X}\to \mathbb{R}$ be a continuously differentiable convex  function. We say that the gradient of $h$ is Lipschitz continuous on $\mathcal{X}$ iff there exists $0<l<+\infty$ such that
$$
\|\nabla h(x_1)-\nabla h (x_2)\|\leq l\|x_1-x_2\|,~~~~\forall x_1,x_2 \in \mathcal{X}.
$$
We say that $ h $ is $ \epsilon $-strongly  convex function with a strong convexity parameter $ \epsilon\geq 0 $ iff $h-\frac{\epsilon}{2}\|\cdot\|^2  $ is a convex function. Clearly, the following gradient inequality holds:
\begin{equation}\label{strong}
\left\langle \nabla h(x_1),x_2-x_1 \right\rangle \leq h(x_2)-h(x_1)-\frac{\epsilon}{2}\|x_1-x_2\|^2, ~~~ \forall x_1, x_2\in \mathcal{X}.
\end{equation}

Now, consider the convex optimization problem with linear equality constraints (\ref{constrained}).The Lagrangian function $\mathcal{L}:\mathcal{X}\times\mathcal{Y}\rightarrow\mathbb{R}$ of problem (\ref{constrained}) is defined by
$$\mathcal{L}(x,\lambda)=f(x)+\langle\lambda,Ax-b\rangle$$
and that $ (x^*,y^*) \in \mathcal{X}\times \mathcal{Y}$ is said to be a saddle point of the  Lagrangian function $ \mathcal{L} $ iff
 \begin{equation}\label{Saddlepoint}
 \mathcal{L}(x^*,\lambda)\leq \mathcal{L}(x^*,\lambda^*) \leq \mathcal{L}(x,\lambda^*),~~\forall (x,\lambda)\in\mathcal{X}\times\mathcal{Y}.
 \end{equation}
The saddle point set of $\mathcal{L}$ is denote by $\Omega$. It is well-known that $(x^*,\lambda^*)\in\Omega$ if and only if
\begin{equation}\left\{\label{KKT}
\begin{split}
&\nabla f(x^*)+A^\top\lambda^*=0,\\
&Ax^*-b=0.
\end{split}\right.
\end{equation}
A pair $(x^*,\lambda^*)\in\Omega$ is also called a primal-dual solution of problem (\ref{constrained}).

For $ c>0 $ and $ p>0 $, associated with the Lagrangian function $ \mathcal{L} $, we introduce   the  augmented  Lagrangian function $ \mathcal{L}_t: \mathcal{X} \times \mathcal{Y} \to \mathbb{R}$ defined as
\begin{equation}\label{asd}
\mathcal{L}_t(x,\lambda):= \mathcal{L}(x,\lambda)  + \frac{c}{2t^p}\left(\|x\|^2-\|\lambda\|^2\right)=f(x)+\left \langle Ax-b,\lambda \right \rangle+\frac{c}{2t^p}\left(\|x\|^2-\|\lambda\|^2\right).
\end{equation}
Obviously, $ \mathcal{L}_t(\cdot,\lambda)$ is $ \frac{c}{t^p} $-strongly convex for any $ \lambda\in \mathcal{Y} $, and $ \mathcal{L}_t(x,\cdot)$ is $ \frac{c}{t^p} $-strongly concave for any $ x\in \mathcal{X} $. This means that $ \mathcal{L}_t$ admits a unique  saddle point $ (x_t,\lambda_t) \in \mathcal{X} \times \mathcal{Y}$, i.e.,
 \begin{equation}\label{Saddlepoint-t}
 \mathcal{L}_t(x_t,\lambda)\leq \mathcal{L}_t(x_t,\lambda_t) \leq \mathcal{L}_t(x,\lambda_t),~~\forall (x,\lambda)\in\mathcal{X}\times\mathcal{Y}.
 \end{equation}
 Naturally, the system of primal-dual optimality conditions reads
 \begin{equation}\label{KKT1}
\left\{ \begin{split}
& 0=\nabla_x\mathcal{L}_t(x_t,\lambda_t) =\nabla  f(x_t)+A^\top\lambda_t+\frac{c}{t^p}x_t,\\
& 0=\nabla_\lambda\mathcal{L}_t(x_t,\lambda_t)=Ax_t-b-\frac{c}{t^p}\lambda_t.
\end{split}\right.
 \end{equation}
Here, $\nabla _x\mathcal{L}_t$ and $\nabla _\lambda\mathcal{L}_t$ denote the gradients of $\mathcal{L}_t$ with respect to the first argument and  the second argument, respectively.

The following important property will be used in the sequel.
\begin{lemma}\textup{\cite[Lemma 6]{He2022}}   \label{le1}
Assume that $g:[t_0,+\infty)\rightarrow\mathcal{X}$ is a continuous differentiable function, $\eta:[t_0,+\infty)\rightarrow[0,+\infty)$ is a continuoud differentiable function, $t_0>0$, and $C\geq0$. If
$$
\left\|g(t)+\int^t_{t_0}\eta(s)g(s)ds\right\|\leq C, ~~\forall t\geq t_0,
$$
then
$
\sup_{t\geq t_0}\|g(t)\|<+\infty.
$
\end{lemma}
\begin{lemma}\textup{\cite[Lemma 2.3]{battahi25}}   \label{xtyt}
Let $ (\bar{x}^*,\bar{\lambda}^*) =\mathrm{Proj}_{\Omega}0 $ and $(x_t,\lambda_t) $ be the saddle point of $ \mathcal{L}_t $. Suppose $0<p<1$. Then,
\begin{itemize}
\item[{\rm (i)}]   $ \lim_{t\to +\infty} \|(x_t,\lambda_t)-(\bar{x}^*,\bar{\lambda}^*)\| =0$ and $ \|(x_t,\lambda_t)\|\leq\|(\bar{x}^*,\bar{\lambda}^*)\| $ for all $ t $ $\geq$ $ t_0 $.
\item[{\rm (ii)}]   $ \|(\dot{x}_t,\dot{\lambda}_t)\|\leq \frac{p}{t}\|(x_t,\lambda_t) \|\leq \frac{p}{t}\|(\bar{x}^*,\bar{\lambda}^*) \| $ for all $ t $ $\geq$ $ t_0 $.
\end{itemize}
\end{lemma}
\begin{remark}
Let $ z_t=(x_t,\lambda_t) $ and $ \bar{z}^*=(\bar{x}^*,\bar{\lambda}^*) $. By Lemma \ref{xtyt}, it is easy to see that
\begin{equation}\label{z1}
\max\left\{\|x_t\|^2,\|\lambda_t\|^2\right\}\leq \|\bar{z}^*\|^2
\end{equation}
and
\begin{equation}\label{z2}
\max\left\{\|\dot{x}_t\|^2,\|\dot{\lambda}_t\|^2\right\}\leq \|\dot{z}_t\|^2\leq \frac{p^2}{t^2} \|\bar{z}^*\|^2.
\end{equation}
\end{remark}
\begin{lemma}\textup{\cite[Lemma 2.4]{battahi25}}   \label{le2}
For any $t\geq t_0$, Suppose $c>0$ and $0<p<1$, it holds,
\begin{equation*}
\frac{d}{dt}\mathcal{L}_t(x_t,\lambda_t)=\frac{cp}{2t^{p+1}}\left(\|\lambda_t\|^2-\|x_t\|^2\right).
\end{equation*}
\end{lemma}

\section{Strong convergence of trajectory to the minimal norm solution}
In this section, we establish a simultaneous result on the strong convergence of the trajectory generated by the system (\ref{dyn1}), and the convergence rate of the primal-dual gap, the objective residual, and the feasibility violation.

Let $ (x ,\lambda ):\left[t_0, +\infty\right)\to\mathcal{X} \times \mathcal{Y}  $ be a global solution of   system \textup{(\ref{dyn1})} and let $ (x_t,\lambda_t) $ be the saddle point of $ \mathcal{L}_t $.  We introduce the   energy function $ \mathcal{E}(t):\left[t_0,+\infty \right) \to \mathbb{R} $  defined as
\begin{equation}\label{defen}
\begin{split}
\mathcal{E}(t)=&a(t)t^q\left(\mathcal{L}_t(x(t),\lambda_t)-\mathcal{L}_t(x_t,\lambda_t)\right)\\
&+\frac{1}{2}\lVert(\alpha-1)(x(t)-x_t)+t^q \big(m(t)\dot{x}(t)+\gamma \nabla_{x}\mathcal{L}_t\left(x(t),\lambda(t)\right)\big)\rVert^2\\
&+\frac{1}{2}b(t)\|x(t)-x_t\|^2+\frac{1}{2}\|\lambda(t)-\lambda_t\|^2,
\end{split}
\end{equation}
where
$
a(t)=m(t)t^{q+s} -2\gamma q m(t) t^{q-1}-\gamma\dot{m}(t)t^{q}+\gamma
$
and
$
b(t)=-(\alpha-1)\big(q m(t)t^{q-1}+\dot{m}(t)t^q-1\big).
$

In the sequel, we will employ the following mild assumptions.  Note that the similar assumptions have been used in \cite{L2024S,csim2024}.

\begin{assumption}\label{assume1}
There exists   $k_1>0$ such that $ \frac{\gamma}{t^{q+s}}\leq m(t)\leq\frac{k_1}{t^{q}}$,  for $t $ big enough.
\end{assumption}
\begin{assumption}\label{assume2}
 There exists $k_2>0$ such that $t\left| \dot{m}(t) \right|\leq {k_2} m(t)$  and  $t^2 \left|\ddot{m}(t)\right|\leq {k_2}m(t)$, for $t$ big enough.
\end{assumption}

Obviously, by $ 0<q<1 $,  $ s>0 $ and Assumptions \textup{\ref{assume1}}  and  \textup{\ref{assume2}}, there exists $ t_1\geq t_0 $ such that   $a(t)\geq0$ and $ b(t)\geq0 $,    $ \forall t\geq t_1 $. Thus, $ \mathcal{E}(t)\geq 0 $,  $ \forall t\geq t_1 $.

The following proposition   gives a estimate  for the  energy function (\ref{defen}), which will play a crucial role  in establishing convergence results.
\begin{proposition}\label{lemma4.21} Suppose that $0<p<1$ and Assumptions \textup{\ref{assume1}}  and  \textup{\ref{assume2}} are satisfied.
 Let $ (x ,\lambda ):\left[t_0, +\infty\right)\to\mathcal{X} \times \mathcal{Y}  $ be a global solution of  system \textup{(\ref{dyn1})} and let $ (x_t,\lambda_t) $ be the saddle point of $ \mathcal{L}_t $.
 Then,  there exists  a nonegative constant $\mathcal{C}_0$ such that for   $ t $ large enough,
\begin{eqnarray*}
\dot{ \mathcal{E} }(t)+\frac{M}{t^r} {\mathcal{E}}(t)
\leq \mathcal{C}_0 \left(m(t)t^{3q+s+p-2}+m(t)t^{q+s-p-1}\right),
\end{eqnarray*}
where $r=\max\{q,p-q-s\}$ and $0<M<\min\left\{\frac{3c}{4(2\alpha-1)},c(\alpha-1)-\frac{1}{a_1}\right\}$.
\end{proposition}
\begin{proof}
Now, we  analyze the time derivative of $ \mathcal{E}(t) $.
Firstly,  from (\ref{asd}) and (\ref{KKT1}), we have
\begin{equation*}
\begin{split}
&\frac{d}{dt}\mathcal{L}_t(x(t),\lambda_t)\\
=&\langle\nabla f(x(t)),\dot{x}(t) \rangle+\langle A^\top \lambda_t,\dot{x}(t) \rangle+\langle Ax(t)-b,\dot{\lambda}_t\rangle\\
&+\frac{c}{t^p}(\langle x(t),\dot{x}(t) \rangle-\langle \lambda_t,\dot{\lambda}_t \rangle) -\frac{1}{2}cpt^{-p-1}\left(\|x(t)\|^2-\|\lambda_t\|^2\right)\\
=&\langle \nabla_x\mathcal{L}_t(x(t),\lambda_t),\dot{x}(t)\rangle+\left\langle Ax(t)-b-\frac{c}{t^p}\lambda_t,\dot{\lambda}_t\right\rangle-\frac{1}{2}cpt^{-p-1}\left(\|x(t)\|^2-\|\lambda_t\|^2\right)\\
=&\langle \nabla_x\mathcal{L}_t(x(t),\lambda_t),\dot{x}(t)\rangle+\langle A(x(t)-x_t),\dot{\lambda}_t\rangle-\frac{1}{2}cpt^{-p-1}\left(\|x(t)\|^2-\|\lambda_t\|^2\right).
\end{split}
\end{equation*}
Then, it follows from Lemma \ref{le2} that
\begin{equation*}
\begin{split}
&\frac{d}{dt}\left(\mathcal{L}_t(x(t),\lambda_t)-\mathcal{L}_t(x_t,\lambda_t)\right)\\
=&\langle\nabla_x\mathcal{L}_t(x(t),\lambda_t),\dot{x}(t)\rangle+\langle A(x(t)-x_t),\dot{\lambda}_t\rangle-\frac{1}{2}cpt^{-p-1}\left(\|x(t)\|^2-\|x_t\|^2\right).
\end{split}
\end{equation*}
Consequently,
\begin{equation}\label{mu112x}
\begin{split}
&\frac{d}{dt}\Big(a(t)t^q\left(\mathcal{L}_t(x(t),\lambda_t)-\mathcal{L}_t(x_t,\lambda_t)\right)\Big)\\
=&(\dot{a}(t)t^q+qa(t)t^{q-1})
\left(\mathcal{L}_t(x(t),\lambda_t)-\mathcal{L}_t(x_t,\lambda_t)\right)+a(t)t^q\Big(\langle \nabla_x\mathcal{L}_t(x(t),\lambda_t),\dot{x}(t)\rangle \\
&\left.+\langle A(x(t)-x_t),\dot{\lambda}_t\rangle-\frac{1}{2}cpt^{-p-1}\left(\|x(t)\|^2-\|x_t\|^2\right)\right).
\end{split}
\end{equation}

Secondly, let $
\vartheta(t)\coloneqq(\alpha-1)(x(t)-x_t)+t^q \Big(m(t)\dot{x}(t)+\gamma \nabla_{x}\mathcal{L}_t\left(x(t),\lambda(t)\right)\Big).
$
Then,
\begin{equation*}
\begin{split}
\dot{\vartheta}(t)=&(\alpha-1)(\dot{x}(t)-\dot{x}_t)+qt^{q-1}\Big(m(t)\dot{x}(t)+\gamma \nabla_{x}\mathcal{L}_t\big(x(t),\lambda(t)\big)\Big)\\
&+t^q\Big(\dot{m}(t)\dot{x}(t)+m(t)\ddot{x}(t)+\gamma\frac{d}{dt} \nabla_{x}\mathcal{L}_t\big(x(t),\lambda(t)\big)\Big).
\end{split}
\end{equation*}
This together with the first equality of   (\ref{dyn1}) yields
\begin{equation*}
\dot{\vartheta}(t)=(qm(t)t^{q-1}+\dot{m}(t)t^q-1)\dot{x}(t)-(\alpha-1)\dot{x}_t+(\gamma qt^{q-1}-t^{q+s})\nabla_{x}\mathcal{L}_t\left(x(t),\lambda(t)\right).
\end{equation*}
Therefore,
\begin{eqnarray}\label{dyn1sun}
\begin{split}
&\left \langle \vartheta(t),\dot{\vartheta}(t)\right \rangle\\
=&(\alpha-1)\big(qm(t)t^{q-1}+\dot{m}(t)t^q-1\big)\left \langle x(t)-x_t,\dot{x}(t)\right \rangle\\
&+m(t)t^q\big(qm(t)t^{q-1}+\dot{m}(t)t^{q}-1\big)\| \dot{x}(t)\|^2 \\
&+\big(2\gamma qm(t)t^{2q-1}+\gamma\dot{m}(t) t^{2q}-\gamma t^q-m(t)t^{2q+s}\big)\left \langle\nabla_x\mathcal{L}_t(x(t),\lambda(t)),\dot{x}(t)\right \rangle\\
&-(\alpha-1)^2\langle x(t)-x_t,\dot{x}_t\rangle-(\alpha-1)m(t)t^q\langle \dot{x}(t),\dot{x}_t\rangle\\
&-(\alpha-1)\gamma t^{q}\langle \nabla_{x}\mathcal{L}_t(x(t),\lambda(t)),\dot{x}_t\rangle\\
&+(\alpha-1)\big(\gamma q t^{q-1}-t^{q+s}\big)\left \langle\nabla_x\mathcal{L}_t(x(t),\lambda(t)),x(t)-x_t\right \rangle\\
&+\gamma t^q\big(\gamma qt^{q-1}- t^{q+s}\big)\|\nabla_x\mathcal{L}_t(x(t),\lambda(t))\|^2.
\end{split}
\end{eqnarray}
Note that $
a(t)=m(t)t^{q+s} -2\gamma q m(t) t^{q-1}-\gamma\dot{m}(t)t^{q}+\gamma
$
and
$
b(t)=-(\alpha-1)\big(q m(t)t^{q-1}+\dot{m}(t)t^q-1\big).
$
Then, it follows from (\ref{dyn1sun}) that
\begin{eqnarray}\label{mu0001}
\begin{split}
\left \langle \vartheta(t),\dot{\vartheta}(t)\right \rangle
=&-b(t)\left \langle x(t)-x_t,\dot{x}(t)\right \rangle-\frac{1}{\alpha-1}m(t)t^qb(t)\| \dot{x}(t)\|^2\\
 &-a(t)t^q\left \langle\nabla_x\mathcal{L}_t(x(t),\lambda(t)),\dot{x}(t)\right \rangle-(\alpha-1)^2\langle x(t)-x_t,\dot{x}_t\rangle\\
 &-(\alpha-1)m(t)t^q\langle \dot{x}(t),\dot{x}_t\rangle-(\alpha-1)\gamma t^{q}\langle \nabla_{x}\mathcal{L}_t(x(t),\lambda(t)),\dot{x}_t\rangle\\
&+(\alpha-1)\big(\gamma q t^{q-1}-t^{q+s}\big)\left \langle\nabla_x\mathcal{L}_t(x(t),\lambda(t)),x(t)-x_t\right \rangle\\
&+\gamma t^q\big(\gamma qt^{q-1}- t^{q+s}\big)\|\nabla_x\mathcal{L}_t(x(t),\lambda(t))\|^2.
\end{split}
\end{eqnarray}
Since
$\mathcal{L}_t(\cdot,\lambda_t)$ is $\frac{c}{t^p}$-strong convex function, we   deduce from $(\ref{strong}) $ that
\begin{equation}\label{stre0}
\left \langle\nabla_x\mathcal{L}_t(x(t),\lambda_t),{x}(t)-x_t\right \rangle \geq\mathcal{L}_t(x(t),\lambda_t)-\mathcal{L}_t(x_t,\lambda_t)+\frac{c}{2t^p}\|x(t)-x_t\|^2.
\end{equation}
Note that there exists $t_2\geq t_1$ such that
\begin{equation}\label{strex0}
\gamma qt^{q-1}-t^{q+s}\leq 0,\forall  t\geq t_2.
\end{equation}
Then, combining
$(\ref{stre0}) $, $(\ref{strex0}) $ and $
\nabla_x\mathcal{L}_t(x(t),\lambda(t))=\nabla_x\mathcal{L}_t(x(t),\lambda_t)+A^\top(\lambda(t)-\lambda_t)
$,
we  deduce from $(\ref{mu0001}) $ that for any $t\geq t_2,$
\begin{equation}\label{labda0}
\begin{split}
&\left \langle \vartheta(t),\dot{\vartheta}(t)\right \rangle\\
\leq&-b(t)\left \langle x(t)-x_t,\dot{x}(t)\right \rangle-\frac{1}{\alpha-1}m(t)t^qb(t)\| \dot{x}(t)\|^2\\
& -a(t)t^q\left \langle\nabla_x\mathcal{L}_t(x(t),\lambda_t),\dot{x}(t)\right \rangle-a(t) t^q\left\langle A^\top(\lambda(t)-\lambda_t),\dot{x}(t)\right\rangle\\
&-(\alpha-1)^2\langle x(t)-x_t,\dot{x}_t\rangle-(\alpha-1)m(t)t^q\langle \dot{x}(t),\dot{x}_t\rangle\\
&-(\alpha-1)\gamma t^{q}\langle \nabla_{x}\mathcal{L}_t(x(t),\lambda(t)),\dot{x}_t\rangle\\
&+(\alpha-1)(\gamma qt^{q-1}-t^{q+s})\bigg(\mathcal{L}_t(x(t),\lambda_t)-\mathcal{L}_t(x_t,\lambda_t)+\frac{c}{2t^p}\|x(t)-x_t\|^2\bigg)\\
&+(\alpha-1)\big(\gamma q t^{q-1}-t^{q+s}\big)\left \langle A^\top(\lambda(t)-\lambda_t),x(t)-x_t\right \rangle\\
&+\gamma t^q(\gamma qt^{q-1}-t^{q+s})\|\nabla_{x}\mathcal{L}_t(x(t),\lambda(t))\|^2.
\end{split}
\end{equation}

 Thirdly,
\begin{equation}\label{labda00}
\begin{split}
&\frac{d}{dt}\left( \frac{1}{2}b(t)\|x(t)-x_t\|^2 \right)\\
=&\frac{1}{2}\dot{b}(t)\|x(t)-x_t\|^2+b(t)\Big(\left \langle x(t)-x_t,\dot{x}(t)\right \rangle-\left \langle x(t)-x_t,\dot{x}_t\right \rangle\Big).
\end{split}
\end{equation}

Fourthly,
\begin{equation*}
\begin{split}
&\frac{d}{dt}\Big(\frac{1}{2}\|\lambda(t)-\lambda_t\|^2\Big)\\
=&\langle\lambda(t)-\lambda_t,\dot{\lambda}(t)-\dot{\lambda}_t\rangle\\
=&(\alpha-1)\big(t^{q+s}-\gamma qt^{q-1}\big)\big\langle\lambda(t)-\lambda_t,\nabla_\lambda \mathcal{L}_t(x(t)+\theta(t)\dot{x}(t),\lambda(t))\rangle-\langle \lambda(t)-\lambda_t,\dot{\lambda}_t\rangle\\
=&(\alpha-1)\big(t^{q+s}-\gamma qt^{q-1}\big)\big\langle\lambda(t)-\lambda_t,\nabla_\lambda\mathcal{L}_t(x_t,\lambda(t))\big\rangle\\
&+(\alpha-1)\big(t^{q+s}-\gamma qt^{q-1}\big)\big\langle\lambda(t)-\lambda_t,A(x(t)+\theta(t)\dot{x}(t)-x_t)\big\rangle-\langle \lambda(t)-\lambda_t,\dot{\lambda}_t\rangle,
\end{split}
\end{equation*}
where the second equality holds due to the second equality of (\ref{dyn1}) and   the last equality holds due to $$\nabla_\lambda \mathcal{L}_t(x(t)+\theta(t)\dot{x}(t),\lambda(t))
=\nabla_\lambda\mathcal{L}_t(x_t,\lambda(t))+A(x(t)+\theta(t)\dot{x}(t)-x_t).$$
Further, from the $\frac{c}{t^p}$-strongly convexity of $-\mathcal{L}_t(x_t,\cdot)$, we deduce that
\begin{equation*}
\begin{split}
-\big\langle\lambda(t)-\lambda_t,\nabla_\lambda\mathcal{L}_t(x_t,\lambda(t))\big\rangle
&\geq\mathcal{L}_t(x_t,\lambda_t)-\mathcal{L}_t(x_t,\lambda(t))+\frac{c}{2t^p}\|\lambda(t)-\lambda_t\|^2\\
&\geq\frac{c}{2t^p}\|\lambda(t)-\lambda_t\|^2.
\end{split}
\end{equation*}
Then,
\begin{equation}\label{labda}
\begin{split}
&\frac{d}{dt}\Big( \frac{1}{2}\|\lambda(t)-\lambda_t\|^2 \Big)\\
\leq & -\frac{c}{2t^p}(\alpha-1)\big(t^{q+s}-\gamma q t^{q-1}\big)\|\lambda(t)-\lambda_t\|^2\\
&+(\alpha-1)\big(t^{q+s}-\gamma q t^{q-1}\big)\big\langle\lambda(t)-\lambda_t,A(x(t)+\theta(t)\dot{x}(t)-x_t)\big\rangle\\
&-\langle \lambda(t)-\lambda_t,\dot{\lambda}_t\rangle\\
= & \frac{1}{2 }c(\alpha-1)\big(\gamma q t^{q-p-1}-t^{q+s-p}\big)\|\lambda(t)-\lambda_t\|^2+a(t)t^q\big\langle A^\top (\lambda(t)-\lambda_t),\dot{x}(t) \big\rangle\\
&+(\alpha-1)\big(t^{q+s}-\gamma q t^{q-1}\big)\big\langle\lambda(t)-\lambda_t,A(x(t)-x_t)\big\rangle\\
&-\langle \lambda(t)-\lambda_t,\dot{\lambda}_t\rangle,
\end{split}
\end{equation}
where the   equality holds due to   (\ref{theta}).

Therefore,  together with (\ref{mu112x}), (\ref{labda0}), (\ref{labda00}) and (\ref{labda}),  we obtain that for any $t\geq t_1,$
\begin{equation*}
\begin{split}
&\dot{\mathcal{E}}(t)\\
\leq&\Big(\dot{a}(t)t^q+q a(t)t^{q-1}+(\alpha-1)\big(\gamma q t^{q-1}-t^{q+s}\big)\big)\big(\mathcal{L}_t(x(t),\lambda_t)-\mathcal{L}_t(x_t,\lambda_t)\Big)\\
&+a(t)t^q\big\langle A(x(t)-x_t),\dot{\lambda}_t\big \rangle -\frac{1}{2}cpa(t)t^{q-p-1}\big(\|x(t)\|^2-\|x_t\|^2\big)\\
&-\frac{1}{\alpha-1}m(t)t^qb(t)\|\dot{x}(t)\|^2-\left((\alpha-1)^2+b(t)\right)\langle x(t)-x_t,\dot{x}_t \rangle\\
&-(\alpha-1)m(t)t^q\left \langle \dot{x}(t),\dot{x}_t\right \rangle-(\alpha-1)\gamma t^q\left \langle \nabla_{x}\mathcal{L}_t(x(t),\lambda(t)),\dot{x}_t \right \rangle\\
&+\frac{1}{2}\left( c(\alpha-1)\big(\gamma q t^{q-p-1}-t^{q+s-p}\big)+\dot{b}(t)\right)  \|x(t)-x_t\|^2  \\
&+\gamma t^q(\gamma qt^{q-1}-t^{q+s})\|\nabla_x\mathcal{L}_t(x(t),\lambda(t))\|^2\\
&+\frac{1}{2 }c(\alpha-1)\big(\gamma q t^{q-p-1}-t^{q+s-p}\big)\|\lambda(t)-\lambda_t\|^2-\langle \lambda(t)-\lambda_t,\dot{\lambda}_t\rangle.
\end{split}
\end{equation*}
Note that
\begin{eqnarray*}
\left\{\begin{split}
a(t)t^q\big \langle A(x(t)-x_t),\dot{\lambda}_t\big \rangle\leq& \frac{1}{2}a(t)t^q\left( \frac{1}{4 }c (\alpha-1)t^{-q-p}\|x(t)-x_t\|^2\right.\\
&\left.+\frac{4 }{c(\alpha-1)}t^{q+p}\|A\|^2\|\dot{\lambda}_t\|^2\right),\\
-\left((\alpha-1)^2+b(t)\right)\big \langle x(t)-x_t,\dot{x}_t \big \rangle
\leq &\frac{1}{2}\left((\alpha-1)^2+b(t)\right)\left( \frac{c}{4\alpha}t^{q+s-p}\|x(t)-x_t\|^2\right. \\
&\left.+\frac{4\alpha}{c}t^{p-q-s}\|\dot{x}_t\|^2\right),\\
-(\alpha-1)m(t)t^q\big \langle \dot{x}(t),\dot{x}_t\big \rangle
\leq&\frac{1}{2}m(t)t^q\left(\|\dot{x}(t)\|^2+ (\alpha-1)^2\|\dot{x}_t\|^2\right),\\
-(\alpha-1)\gamma t^q\left \langle \nabla_x\mathcal{L}_t(x(t),\lambda(t)),\dot{x}_t\right \rangle
\leq&\frac{1}{2}\gamma t^q \Big(t^{q+s}\|\nabla_x\mathcal{L}_t(x(t),\lambda(t))\|^2\\
&+ (\alpha-1)^2t^{-q-s}\|\dot{x}_t\|^2 \Big),\\
-\left \langle {\lambda}(t)-\lambda_t,\dot{\lambda}_t\right \rangle
\leq&\frac{1}{2}\left(\frac{1}{a_1}t^{q+s-p}\|{\lambda}(t)-\lambda_t\|^2+ {a_1}{t^{p-q-s}}\|\dot{\lambda}_t\|^2\right).
\end{split}
\right.
\end{eqnarray*}
Here  $a_1>0$ is an  arbitrary constant.
Thus, for any $t\geq t_1,$
\begin{equation}\label{M}
\begin{split}
&\dot{ \mathcal{E} }(t) \\
\leq&\Big(\dot{a}(t)t^q+qa(t)t^{q-1}+(\alpha-1)\big(\gamma q t^{q-1}-t^{q+s}\big) \Big)\big(\mathcal{L}_t(x(t),\lambda_t)-\mathcal{L}(x_t,\lambda_t)\big)\\
&+\frac{1}{2}\left(c(\alpha-1)\big(\gamma q t^{q-p-1}-t^{q+s-p}\big)+ \dot{b}(t) +\frac{1}{4}c(\alpha-1) a(t)t^{-p} \right.\\
&~~~\left.+\frac{1}{4\alpha}c (\alpha-1)^2 t^{q+s-p}+\frac{1}{4\alpha}cb(t) t^{q+s-p}\right)\|x(t)-x_t\|^2\\
&+\left(\frac{2}{c(\alpha-1)}\|A\|^2a(t)t^{2q+p}+\frac{1}{2}a_1 t^{p-q-s}\right)\|\dot{\lambda}_t\|^2\\
&+\frac{1}{2}cp a(t)t^{q-p-1}\big(\|x_t\|^2-\|x(t)\|^2\big)+\left( -\frac{1}{\alpha-1}b(t)+\frac{1}{2}\right)m(t)t^{q}\| \dot{x}(t)\|^2\\
&+\left(\frac{2\alpha}{c}(\alpha-1)^2 t^{p-q-s}+\frac{2\alpha}{c}b(t)t^{p-q-s}+\frac{1}{2}(\alpha-1)^2m(t)t^q+\frac{1}{2 }\gamma(\alpha-1)^2 {t^{-s}}\right)\|\dot{x}_t\|^2\\
&+\gamma t^q\left(\gamma q t^{q-1}-\frac{1}{2}t^{q+s} \right)\|\nabla_x\mathcal{L}_t(x(t),\lambda(t)\|^2\\
&+\frac{1}{2}\left(c(\alpha-1)(\gamma q t^{q-p-1}-t^{q+s-p}) +\frac{ 1}{a_1}t^{q+s-p}\right)\|\lambda(t)-\lambda_t\|^2.
\end{split}
\end{equation}
  On the other hand, note that
\begin{equation*}
\begin{split}
&\frac{1}{2}\|(\alpha-1)(x(t)-x_t)+t^q \left(m(t)\dot{x}(t)+\gamma \nabla_{x}\mathcal{L}_t\left(x(t),\lambda(t)\right)\right)\|^2\\
\leq& (\alpha-1)^2\|x(t)-x_t\|^2+2{m^2(t)}t^{2q}\|\dot{x}(t)\|^2+2\gamma^2t^{2q}\|\nabla_{x}\mathcal{L}_t\left(x(t),\lambda(t)\right)\|^2.
\end{split}
\end{equation*}
Then, it follows from  (\ref{defen}) that
\begin{equation*}
\begin{split}
 \mathcal{E}(t)
\leq& a(t) t^{q }\left(\mathcal{L}_t(x(t),\lambda_t)-\mathcal{L}_t(x_t,\lambda_t)\right)
+ \frac{1}{2}\left( 2 (\alpha-1)^2+ b(t)  \right)   \|x(t)-x_t\|^2\\
&+2 {m^2(t)}t^{2q }\| \dot{x}(t)\|^2+2  \gamma^2t^{2q }\|\nabla_x\mathcal{L}_t(x(t),\lambda(t))\|^2+\frac{1}{2}  \|\lambda(t)-\lambda_t\|^2.
\end{split}
\end{equation*}
This  together with (\ref{M})  yields  for any $ t\geq t_1 $,
\begin{eqnarray}\label{pro04100}
\begin{split}
&\dot{ \mathcal{E} }(t)+\frac{M}{t^r} {\mathcal{E}}(t)\\
\leq&\Big(\dot{a}(t)t^q+qa(t)t^{q-1}+(\alpha-1)\big(\gamma q t^{q-1}-t^{q+s}\big)+Ma(t)t^{q-r}\Big)\big(\mathcal{L}_t(x(t),\lambda_t)-\mathcal{L}(x_t,\lambda_t)\big)\\
&+\frac{1}{2}\left(c(\alpha-1)\big(\gamma q t^{q-p-1}-t^{q+s-p}\big)+ \dot{b}(t)+\frac{1}{4}c(\alpha-1) a(t)t^{-p} +\frac{1}{4\alpha}c (\alpha-1)^2 t^{q+s-p}\right.\\
&~~~\left.+\frac{1}{4\alpha}cb(t) t^{q+s-p}+2M(\alpha-1)^2t^{-r} +M b(t)t^{-r}\right)\|x(t)-x_t\|^2\\
&+\left(\frac{2}{c(\alpha-1)}\|A\|^2a(t)t^{2q+p}+\frac{1}{2}a_1 t^{p-q-s}\right)\|\dot{\lambda}_t\|^2+\frac{1}{2}cp a(t)t^{q-p-1}\big(\|x_t\|^2-\|x(t)\|^2\big)\\
&+\left(2Mm(t)t^{q-r}-\frac{1}{\alpha-1}b(t)+\frac{1}{2}\right)m(t)t^{q}\| \dot{x}(t)\|^2\\
&+\left(\frac{2\alpha}{c}(\alpha-1)^2 t^{p-q-s}+\frac{2\alpha}{c}b(t)t^{p-q-s}+\frac{1}{2}(\alpha-1)^2m(t)t^q+\frac{1}{2 }\gamma(\alpha-1)^2 {t^{-s}}\right)\|\dot{x}_t\|^2\\
&+\gamma t^q\left(\gamma q t^{q-1}-\frac{1}{2}t^{q+s}+2M\gamma t^{q-r}\right)\|\nabla_x\mathcal{L}_t(x(t),\lambda(t)\|^2\\
&+\frac{1}{2}\left(c(\alpha-1)(\gamma q t^{q-p-1}-t^{q+s-p})+\frac{ 1}{a_1}t^{q+s-p}+ M t^{-r}\right)\|\lambda(t)-\lambda_t\|^2,
\end{split}
\end{eqnarray}
where    $r=\max\{q,p-q-s\}$ and $0<M<\min\left\{\frac{3c}{4(2\alpha-1)},c(\alpha-1)-\frac{1}{a_1}\right\}$.

Now, we analyze the coefficients on the right hand side of  (\ref{pro04100}).

$\mathbf{\textup{(i)}}$. Note that  $
a(t)=m(t)t^{q+s} -2\gamma q m(t) t^{q-1}-\gamma\dot{m}(t)t^{q}+\gamma
$ and
$
\dot{a}(t)=\dot{m}(t)t^{q+s}+(q+s)m(t)t^{q+s-1}-3\gamma q\dot{m}(t)t^{q-1}
-2\gamma q(q-1) {m}(t)t^{q-2}-\gamma \ddot{m}(t)t^{q}.
$
 Then, it is easy to show that
\begin{eqnarray}\label{41the1x}
\begin{split}
&\dot{a}(t)t^q+qa(t)t^{q-1}+(\alpha-1)\big(\gamma q t^{q-1}-t^{q+s}\big)+Ma(t)t^{q-r}\\
=&-2\gamma q(2q-1)m(t)t^{2q-2}-2q\gamma Mm(t) t^{2q-r-1}+(2q+s)m(t)t^{2q+s-1}\\
&+Mm(t)t^{2q+s-r}-4\gamma q\dot{m}(t)t^{2q-1}- M\gamma\dot{m}(t)t^{2q-r}+\dot{m}(t)t^{2q+s}\\
&-\gamma\ddot{m}(t)t^{2q}
-(\alpha-1)t^{q+s}+ \alpha  \gamma q t^{q-1}+M\gamma t^{q-r}.
\end{split}
\end{eqnarray}
From    Assumptions  \textup{\ref{assume1}}  and  \textup{\ref{assume2}}, we   demonstrate that,  for $t$ big enough,
\begin{equation*}
\left\{\begin{split}
&m(t)t^{2q-2}<m(t)t^{2q-r-1}<m(t)t^{2q+s-1}<m(t)t^{2q+s-r}\leq k_1t^{q+s-r},\\
&|\dot{m}(t)|t^{2q-1}<|\dot{m}(t)|t^{2q-r}<|\dot{m}(t)|t^{2q+s}\leq k_1k_2t^{q+s-1},\\
&|\ddot{m}(t)|t^{2q}<k_1k_2t^{q-2}.
\end{split}\right.
\end{equation*}
Then, together with $0\leq q\leq r<1$, $q-2<q+s-1<q+s-r<q+s$ and $ \alpha>1 $, we   deduce from (\ref{41the1x}) that there exist  $ t_2\geq t_1 $ and $ {c}_1>0$ such that
\begin{equation*}
\dot{a}(t)t^q+qa(t)t^{q-1}+(\alpha-1)\big(\gamma q t^{q-1}-t^{q+s}\big)+Ma(t)t^{q-r}
\leq -c_1t^{q+s}<0, \ \ \forall t\geq t_2.
\end{equation*}

$\mathbf{\textup{(ii)}}$. Consider the coefficient of $\|x(t)-x_t\|^2$.  Let
\begin{eqnarray*}
\begin{split}
l(t)=&c(\alpha-1)\big(\gamma q t^{q-p-1}-t^{q+s-p}\big)+ \dot{b}(t)+\frac{1}{4}c(\alpha-1) a(t)t^{-p}\\
&+\frac{1}{4\alpha}c (\alpha-1)^2 t^{q+s-p}
+\frac{1}{4\alpha}cb(t) t^{q+s-p}+2M(\alpha-1)^2t^{-r} +M b(t)t^{-r} .
\end{split}
\end{eqnarray*}
Note that $
a(t)=m(t)t^{q+s} -2\gamma q m(t) t^{q-1}-\gamma\dot{m}(t)t^{q}+\gamma
$,  $
b(t)=-(\alpha-1)\big(q m(t)t^{q-1}+\dot{m}(t)t^q-1\big)
$ and $\dot{b}(t)=-(\alpha-1)(2q\dot{m}(t)t^{q-1}+q(q-1)m(t)t^{q-2}+\ddot{m}(t)t^q)$. Then, $l(t)$ can be rewritten as
\begin{equation*}
\begin{split}
l(t)=&c(\alpha-1) \gamma q t^{q-p-1} -\frac{3}{4}c(\alpha-1) t^{q+s-p}-2 q(\alpha-1)\dot{m}(t)t^{q-1}\\
&- q(q-1)(\alpha-1) m(t)t^{q-2}
- (\alpha-1) \ddot{m}(t)t^q\\
&+\frac{1}{4}c(\alpha-1)\Big(m(t)t^{q+s-p}-2\gamma qm(t)t^{q-p-1}-\gamma  \dot{m}(t)t^{q-p}+\gamma t^{-p}\Big) \\
&-\frac{1}{4\alpha}c(\alpha-1)\Big( qm(t) t^{2q+s-p-1}+ \dot{m}(t)t^{2q+s-p}\Big)\\
&+\Big(2M(\alpha-1)^{2}+M(\alpha-1)\Big)t^{-r}- M(\alpha-1)q m(t)t^{q-r-1}- M(\alpha-1) \dot{m}(t)t^{q-r}.
\end{split}
\end{equation*}
This together with $ \alpha>1 $ and $m(t)\geq 0 $ yields
\begin{equation}\label{41the400l}
\begin{split}
l(t)
\leq&c(\alpha-1) \gamma q t^{q-p-1}-\frac{3}{4}c(\alpha-1) t^{q+s-p}-2 q(\alpha-1)\dot{m}(t)t^{q-1}\\
&- q(q-1)(\alpha-1) m(t)t^{q-2}
- (\alpha-1) \ddot{m}(t)t^q\\
&+\frac{1}{4}c(\alpha-1)\Big(m(t)t^{q+s-p}-\gamma  \dot{m}(t)t^{q-p}+\gamma t^{-p}\Big) -\frac{1}{4\alpha}c(\alpha-1) \dot{m}(t)t^{2q+s-p}\\
&+\Big(2M(\alpha-1)^{2}+M(\alpha-1)\Big)t^{-r}- M(\alpha-1) \dot{m}(t)t^{q-r}.
\end{split}
\end{equation}
 From    Assumptions  \textup{\ref{assume1}}  and  \textup{\ref{assume2}}, we   demonstrate that,  for $t$ big enough,
\begin{equation}\label{41the400}
\left\{\begin{split}
&|\dot{m}(t)|t^{q-1}<|\dot{m}(t)|t^{q-r}<k_1k_2t^{-r-1},~m(t)t^{q-2}<k_1t^{-2},~|\ddot{m}(t)|t^{q}<k_1k_2t^{-2},\\
&m(t)t^{q+s-p}<k_1t^{s-p},~|\dot{m}(t)|t^{q-p}<k_1k_2t^{-p-1},~|\dot{m}(t)|t^{2q+s-p}<k_1k_2t^{q+s-p-1}.
\end{split}\right.
\end{equation}
Moreover, from $ 0\leq q \leq r <1 $, $ s>0,$  $p>0 $ and  $ r \geq p-q-s$, we have
\begin{equation}\label{41the4010}
 -2<-r-1<q+s-p\mbox{ and } -p-1<q+s-p-1<s-p<q+s-p.
\end{equation}
Combining $\mathrm{(\ref{41the400})}$, $\mathrm{(\ref{41the4010})}$,   $q-p-1<-p<q+s-p$, $r=\max\{q,p-q-s\}$ and $0<M<\frac{3c}{4(2\alpha-1)}$,
 we   deduce from $\mathrm{(\ref{41the400l})}$ that there exist  $ t_3\geq t_2$ and $c_2>0$ such that
$$l(t)\leq -c_2 t^{q+s-p}<0,\ \  \forall t\geq t_3.$$

$\mathbf{\textup{(iii)}}$.  Let us examine the coefficient of $  \|\dot{x}(t)\|^2$. Clearly, for $t$ big enough, $
m(t)t^{q-1}<m(t)<k_1t^{-q},$ $|\dot{m}(t)|t^{q}<k_1k_2t^{-2}$.
Together with $-2<-q<0$ and $ m(t)\geq 0 $, there exists $ t_4\geq t_3$ such that
\begin{equation*}
\begin{split}
 &m(t)t^{q}\left(2M m(t)-\frac{1}{\alpha-1}b(t)+\frac{1}{2}\right)\\
=&m(t)t^{q}\left(2M m(t)+qm(t)t^{q-1}+\dot{m}(t)t^q-\frac{1}{2}\right)\leq0,\  \forall t\geq t_4.
\end{split}
\end{equation*}

$\mathbf{\textup{(iv)}}$.  We consider the coefficient of $\|\nabla_x\mathcal{L}_t(x(t),\lambda(t)\|^2$.  Indeed, from $q-1<q-r<q+s$,  there exists $ t_5\geq t_4$ such that
\begin{equation*}
\gamma t^q\left(\gamma q t^{q-1}-\frac{1}{2}t^{q+s}+2M\gamma t^{q-r}\right)
 \leq0, \ \forall t\geq t_5.
\end{equation*}

$\mathbf{\textup{(v)}}$. Let us now examine  the coefficient of  $\|\lambda(t)-\lambda_t\|^2.$ Indeed, let $a_1>\frac{1}{c(\alpha-1) }$. Note that $ q-p-1<q+s-p
$, $r=\max\{q,p-q-s\}$  and $0<M<c(\alpha-1)-\frac{1}{a_1}$.
Then, there exist  $ t_6\geq t_5$ and $c_3>0$ such that
\begin{equation*}
\frac{1}{2}\left(c(\alpha-1)(\gamma q t^{q-p-1}-t^{q+s-p})+\frac{ 1}{a_1}t^{q+s-p}+ M t^{-r}\right)\leq-c_3 t^{q+s-p}<0 ,\ \forall t\geq t_6.
\end{equation*}

Now, by virtue of $\mathbf{\textup{(i)}}$-$\mathbf{\textup{(v)}}$, it follows from (\ref{pro04100}) that for any $t\geq t_6,$
\begin{eqnarray*}
\begin{split}
&\dot{ \mathcal{E} }(t)+\frac{M}{t^r} {\mathcal{E}}(t)\\
\leq& \left(\frac{2}{c(\alpha-1)}\|A\|^2a(t)t^{2q+p}+\frac{1}{2}a_1 t^{p-q-s}\right)\|\dot{\lambda}_t\|^2+\frac{1}{2}cp a(t)t^{q-p-1}\|x_t\|^2\\
&+\left(\frac{2\alpha}{c}(\alpha-1)^2 t^{p-q-s}+\frac{2\alpha}{c}b(t)t^{p-q-s}+\frac{1}{2}(\alpha-1)^2m(t)t^q+\frac{1}{2 }\gamma(\alpha-1)^2 {t^{-s}}\right)\|\dot{x}_t\|^2.
\end{split}
\end{eqnarray*}
 Let $ \bar{z}^*=(\bar{x}^*,\bar{\lambda}^*)  =\mathrm{Proj}_{\Omega}0$. By virtue of (\ref{z1}) and (\ref{z2}), we have
\begin{eqnarray}\label{the4001}
\dot{ \mathcal{E} }(t)+\frac{M}{t^r} {\mathcal{E}}(t)
\leq \Delta(t) \|\bar{z}^*\|^2,
\end{eqnarray}
where\begin{eqnarray}\label{the41}
\begin{split}
\Delta(t)
=& \frac{2p^2}{c(\alpha-1)}\|A\|^2a(t)t^{2q+p-2}+\frac{1}{2}a_1p^2 t^{p-q-s-2} +\frac{1}{2}cp a(t)t^{q-p-1} \\
&+ \frac{2\alpha}{c}p^2(\alpha-1)^2 t^{p-q-s-2}+\frac{2\alpha }{c}p^2 b(t)t^{p-q-s-2} \\
& +\frac{1}{2}(\alpha-1)^2p^2m(t)t^{q-2}+\frac{1}{2 }\gamma p^2(\alpha-1)^2 {t^{-s-2}}.
\end{split}
\end{eqnarray}
 Clearly, $(\ref{the41})$ can be written as
\begin{equation*}
\begin{split}
 \Delta(t)=&\frac{2p^2}{c(\alpha-1)}\|A\|^2 \left(m(t)t^{3q+p+s-2}-2\gamma qm(t)t^{3q+p-3}-\gamma\dot{m}(t)t^{3q+p-2}+\gamma t^{2q+p-2}\right)\\
 &+ \frac{1}{2}cp\left(m(t)t^{2q+s-p-1}-2\gamma q m(t) t^{2q-p-2}-\gamma\dot{m}(t)t^{2q-p-1}+\gamma t^{q-p-1}\right)\\
&+\left(\frac{2\alpha^2}{c}p^2(\alpha-1) +\frac{1}{2}a_1p^2\right)t^{p-q-s-2}-\frac{2\alpha}{c}(\alpha-1)p^2q m(t)t^{p-s-3}\\
&-\frac{2\alpha}{c}(\alpha-1)p^2\dot{m}(t)t^{p-s-2}+\frac{1}{2}(\alpha-1)^2p^2m(t)t^{q-2}+\frac{1}{2}\gamma p^2(\alpha-1)^2t^{-s-2}.
\end{split}
\end{equation*}
This follows that
\begin{equation*}
\begin{split}
 \Delta(t)
\leq & \frac{2p^2}{(\alpha-1)c}\|A\|^2 \left(m(t)t^{3q+p+s-2}-\gamma\dot{m}(t)t^{3q+p-2}+\gamma t^{2q+p-2}\right)\\
&+\frac{1}{2}cp\left(m(t)t^{2q+s-p-1}-\gamma\dot{m}(t)t^{2q-p-1}+\gamma t^{q-p-1}\right)\\
&+\left(\frac{2\alpha^2}{c}p^2(\alpha-1)+\frac{1}{2}a_1p^2\right)t^{p-q-s-2}-\frac{2\alpha}{c}(\alpha-1)p^2\dot{m}(t)t^{p-s-2} \\
&
+\frac{1}{2}(\alpha-1)^2p^2m(t)t^{q-2}+\frac{1}{2}\gamma p^2(\alpha-1)^2t^{-s-2}.
\end{split}
\end{equation*}
From Assumptions \ref{assume1} and \ref{assume2}, it is easy to verify that for $t$ big enough,
\begin{equation*}\left\{
\begin{split}
& -\gamma \dot{m}(t) t^{3q+p-2}< m(t)t^{3q+s+p-2},~\gamma t^{2q+p-2}<m(t)t^{3q+s+p-2},\\
&-\gamma \dot{m}(t) t^{2q-p-1}  <m(t)t^{2q+s-p-1}, ~\gamma t^{q-p-1}<m(t)t^{2q+s-p-1},\\
&-\dot{m}(t)t^{p-s-2} <m(t)t^{3q+s+p-2},~m(t)t^{q-2}<m(t)t^{3q+s+p-2}.
\end{split}\right.
\end{equation*}
This together with
$2q+p-2>p-q-s-2$ and $q-p-1>-s-2 $ yields there exist   $ t_7\geq t_6$ and a constant $\mathcal{C}_0>0$ such that
\begin{equation*}
\Delta(t)\|\bar{z}^*\|^2\leq \mathcal{C}_0 \left(m(t)t^{3q+s+p-2}+m(t)t^{2q+s-p-1}\right),~\forall~t\geq t_7.
\end{equation*}
Consequently, $(\ref{the4001})$ leads to
\begin{equation}\label{the41sx}
\dot{\mathcal{E}}(t)+\frac{M}{t^r}\mathcal{E}(t)\leq  \mathcal{C}_0 \left(m(t)t^{3q+s+p-2}+m(t)t^{2q+s-p-1}\right),~\forall t\geq t_7.
\end{equation}
The proof is complete. \qed
\end{proof}

Now, we  demonstrate  that both    the   convergence rates of the primal-dual gap,   the objective
residual, and the feasibility violation, as well as the strong convergence of the trajectory, can be achieved simultaneously.
\begin{theorem}\label{slow} Suppose that Assumptions \textup{\ref{assume1}}  and  \textup{\ref{assume2}} are satisfied. Let $ (x ,\lambda ):\left[t_0, +\infty\right)\to\mathcal{X} \times \mathcal{Y}  $ be a global solution of   system \textup{(\ref{dyn1})} and let $ (x_t,\lambda_t) $ be the saddle point of $ \mathcal{L}_t $. Suppose that  $ 0<p<1-q$,    $ 4q+s+p-2<0$  and  $3q+s-p-1<0.$ Then, for
$( {x}^*, {\lambda}^*)=\textup{Proj}_{\Omega}{0} $, we have
\begin{enumerate}
\item[{\rm (i)}] $(x(t),\lambda(t))$ converges strongly to $( {x}^*, {\lambda}^*)$, that is $
\lim_{t\to +\infty} \|(x(t),\lambda(t))-( {x}^*, {\lambda}^*)\| =0.
$
\item[{\rm (ii)}] If $ \frac{1-q}{2}\leq p<1-q$,   as $t\to +\infty$ it holds
\begin{eqnarray*}
\left\{\begin{split}
&\mathcal{L}(x(t),\lambda^*)-\mathcal{L}(x^*,\lambda^*)=\mathcal{O}\left( \sqrt{m(t)}t^{\frac{3q+s+p-2+r}{2}}+\frac{1}{t^p}\right), \\
&|f(x(t))-f(x^*)|=\mathcal{O}\left( \sqrt{m(t)}t^{\frac{3q+s+p-2+r}{2}}+\frac{1}{t^p}\right),\\
&\|Ax(t)-b\|=\mathcal{O}\left( \sqrt{m(t)}t^{\frac{3q+s+p-2+r}{2}}+\frac{1}{t^p}\right),\\
&\|(x(t)-x_t,\lambda(t)-\lambda_t)\|^2= \mathcal{O}\left( m(t)t^{3q+s+p-2+r}\right).
\end{split}\right.
\end{eqnarray*}

\item[{\rm (iii)}] If $0<p<\frac{1-q}{2} $, as $t\to +\infty$ it holds
\begin{eqnarray*}
\left\{\begin{split}
&\mathcal{L}(x(t),\lambda^*)-\mathcal{L}(x^*,\lambda^*)=\mathcal{O}\left( \sqrt{m(t)}t^{\frac{2q+s-p-1+r}{2}}+\frac{1}{t^p}\right),\\
&|f(x(t))-f(x^*)|=\mathcal{O}\left( \sqrt{m(t)}t^{\frac{2q+s-p-1+r}{2}}+\frac{1}{t^p}\right),\\
& \|Ax(t)-b\|=\mathcal{O}\left( \sqrt{m(t)}t^{\frac{2q+s-p-1+r}{2}}+\frac{1}{t^p} \right),\\
& \|(x(t)-x_t,\lambda(t)-\lambda_t)\|^2= \mathcal{O}\left( m(t)t^{2q+s-p-1+r}\right).
\end{split}\right.
\end{eqnarray*}
\end{enumerate}
Here $r=\max\{q,p-q-s\}$.
\end{theorem}
\begin{proof}
We start with the energy function $ \mathcal{E}(t):\left[t_0,+\infty \right) \to \mathbb{R} $ defined  as (\ref{defen}).
Multiplying $ e^{\frac{M}{1-r}t^{1-r}} $ on both sides of (\ref{the41sx}), we have for any $t\geq t_7,$
\begin{equation}\label{define}
\frac{d}{dt}\left(e^{\frac{M}{1-r}t^{1-r}}\mathcal{E}(t) \right)\leq \mathcal{C}_0 \left(m(t)t^{3q+s+p-2}+m(t)t^{2q+s-p-1}\right)e^{\frac{M}{1-r}t^{1-r}}.
\end{equation}
Here $r=\max\{q,p-q-s\}$ and $0<M<\min\left\{\frac{3c}{4(2\alpha-1)},c(\alpha-1)-\frac{1}{a_1}\right\}$.

Now,  we consider the following two cases:

\textbf{Case I}: $ \frac{1-q}{2}\leq p<1-q$. In this case,  $ 3q+s+p -2\geq 2q+s-p-1 $. Consequently, from $(\ref{define})$,  there exist $\mathcal{C}'_0 \geq 0$ and  $ t_{8}\geq t_7 $ such that for any $t\geq t_8,$
\begin{equation}\label{define0}
\frac{d}{dt}\left(e^{\frac{M}{1-r}t^{1-r}}\mathcal{E}(t) \right)\leq \mathcal{C}'_0 m(t)t^{3q+s+p-2}e^{\frac{M}{1-r}t^{1-r}}.
\end{equation}
Integrating $(\ref{define0})$  over $ [t_8, t] $ where $t\geq t_8$, we have
\begin{equation}\label{dt-1-2}
\begin{split}
&e^{\frac{M}{1-r}t^{1-r}}\mathcal{E}(t)
\\ \leq &e^{\frac{M}{1-r}t_8^{1-r}}\mathcal{E}(t_8)+\mathcal{C}'_0\int_{t_8}^{t} m(\omega)\omega^{3q+s+p-2} e^{\frac{M}{1-r}\omega^{1-r}}d\omega.
\end{split}
\end{equation}
We consider the right hand side of $(\ref{dt-1-2})$.
Clearly, there exist $ t_9\geq t_8 $ and $ \mathcal{C}_1>0 $ such that for any $t\geq t_9$,
\begin{equation}\label{dt-1-20}
\begin{split}
&\frac{d}{dt}\left( t^{3q+s+p-2+r}e^{\frac{M}{1-r}t^{1-r}}\right)\\
=&\left((3q+s+p-2+r)t^{3q+s+p-3+r}+Mt^{3q+s+p-2}\right)e^{\frac{M}{1-r}t^{1-r}}\\
\geq& \mathcal{C}_1 t^{3q+s+p-2}e^{\frac{M}{1-r}t^{1-r}}.
\end{split}
\end{equation}
From $(\ref{dt-1-20})$ and $t |\dot{m}(t)|\leq {k_2} m(t)$, we have
\begin{equation}\label{dt-1-2001}
{\begin{split}
 &\int_{t_8}^{t}  m(\omega) \omega^{3q+s+p-2}e^{\frac{M}{1-r}\omega^{1-r}} d\omega \\
 \leq &\int_{t_8}^{t_9}  m(\omega) \omega^{3q+s+p-2}e^{\frac{M}{1-r}\omega^{1-r}} d\omega  +\frac{1}{\mathcal{C}_1} \int_{t_9}^{t}  m(\omega) \frac{d}{d\omega}\left( \omega^{3q+s+p-2+r}e^{\frac{M}{1-r}\omega^{1-r}}\right)d\omega\\
 =&\mathcal{N}_1 +\frac{1}{\mathcal{C}_1}  m(t)  t^{3q+s+p-2+r}e^{\frac{M}{1-r}t^{1-r}}-\frac{1}{\mathcal{C}_1}\int_{t_9}^{t} \dot{m}(\omega) \omega^{3q+s+p-2+r}e^{\frac{M}{1-r}\omega^{1-r}}d\omega\\
 \leq &\mathcal{N}_1+\frac{1}{\mathcal{C}_1} m(t)  t^{3q+s+p-2+r}e^{\frac{M}{1-r}t^{1-r}}+\frac{k_2 }{\mathcal{C}_1}\int_{t_9}^{t}m(\omega)\omega^{3q+s+p-3+r}e^{\frac{M}{1-r}\omega^{1-r}}d\omega,
 \end{split}}
 \end{equation}
 where $\mathcal{N}_1  =\int_{t_8}^{t_9}  m(\omega) \omega^{3q+s+p-2}e^{\frac{M}{1-r}\omega^{1-r}} d\omega-\frac{1}{\mathcal{C}_1} m(t_9)t_{9}^{3q+s+p-2+r}e^{\frac{M}{1-r}t_9^{1-r}}.$

By $r<1$, we have $3q+s+p-3+r<3q+s+p-2$. Then, there exists $t_{10}\geq t_{9}$ such that
  $\frac{k_2}{\mathcal{C}_1}m(t)t^{3q+s+p-3+r} \leq  \frac{1}{2}m(t)t^{3q+s+p-2} ,$ $\forall~t\geq t_{10}.$
Thus,
 \begin{equation}\label{dt-1-2002}
 \begin{split}
 &\frac{k_2}{\mathcal{C}_1}\int_{t_9}^{t}m(\omega)\omega^{3q+s+p-3+r}e^{\frac{M}{1-r}\omega^{1-r}}d\omega\\
 \leq & \frac{k_2}{\mathcal{C}_1}\int_{t_9}^{t_{10}}m(\omega)\omega^{3q+s+p-3+r}e^{\frac{M}{1-r}\omega^{1-r}}d\omega+\frac{1}{2}\int_{ t_{10}}^{t }m(\omega)\omega^{3q+s+p-2}e^{\frac{M}{1-r}\omega^{1-r}}d\omega\\
 =&\mathcal{N}_2+\frac{1}{2}\int_{t_8}^{t} m(\omega) \omega^{3q+s+p-2}e^{\frac{M}{1-r}\omega^{1-r}}d\omega,
 \end{split}
 \end{equation}
where $\mathcal{N}_2=\frac{k_2}{\mathcal{C}_1}\int_{t_9}^{t_{10}}m(\omega)\omega^{3q+s+p-3+r}e^{\frac{M}{1-r}\omega^{1-r}}d\omega-\frac{1}{2}\int_{ t_8}^{t_{10} } m(\omega) \omega^{3q+s+p-2}e^{\frac{M}{1-r}\omega^{1-r}}d\omega$.
Combining (\ref{dt-1-2001}) and (\ref{dt-1-2002}),  we obtain that  for all $t\geq  t_{10}$,
\begin{equation}\label{dt-1-2003}
\begin{split}
\int_{t_8}^{t}  m(\omega) \omega^{3q+s+p-2}e^{\frac{M}{1-r}\omega^{1-r}} d\omega
\leq 2(\mathcal{N}_1+\mathcal{N}_2)+\frac{2}{\mathcal{C}_1} m(t) t^{3q+s+p-2+r}e^{\frac{M}{1-r}t^{1-r}}.
\end{split}
\end{equation}
Then,   it follows from  $(\ref{dt-1-2})$ and $(\ref{dt-1-2003})$ that
\begin{equation}\label{dt-1-2345}
\begin{split}
&e^{\frac{M}{1-r}t^{1-r}}\mathcal{E}(t)\\
\leq &e^{\frac{M}{1-r}t_8^{1-r}}\mathcal{E}(t_8)+\mathcal{C}'_0\left(2(\mathcal{N}_1+\mathcal{N}_2)+\frac{2}{\mathcal{C}_1} m(t) t^{3q+s+p-2+r}e^{\frac{M}{1-r}t^{1-r}}\right).
\end{split}
\end{equation}
Let $\mathcal{N}_4:=e^{\frac{M}{1-r}t_8^{1-r}}\mathcal{E}(t_8)+ \mathcal{C}'_0(2(\mathcal{N}_1+\mathcal{N}_2))$ and $\mathcal{N}_5:=\frac{2\mathcal{C}'_0}{\mathcal{C}_1}$.
Together with $(\ref{dt-1-2345})$,  there exist  $t_{11}\geq t_{10}$ and $\mathcal{C}_3>0$ such that
\begin{equation*}
\mathcal{E}(t)
\leq\frac{\mathcal{N}_4}{e^{\frac{M}{1-r}t^{1-r}} }+\mathcal{N}_5 m(t)t^{3q+s+p-2+r}\leq \mathcal{C}_3 m(t)t^{3q+s+p-2+r},
~\forall~t\geq t_{11}.
\end{equation*}
According to  $(\ref{defen})$, we have
\begin{equation}\label{L2}
\mathcal{L}_t(x(t),\lambda_t)-\mathcal{L}_t(x_t,\lambda_t)= \mathcal{O}\left( t^{q+p-2+r}\right),~~~\mbox{as}~t\to +\infty,
\end{equation}
and
\begin{equation}\label{r1}
\|(x(t)-x_t,\lambda(t)-\lambda_t)\|^2= \mathcal{O}\left( m(t)t^{3q+s+p-2+r}\right),~~~\mbox{as}~t\to +\infty.
\end{equation}
By $(\ref{KKT1})$ and $(\ref{z1})$, it is easy to show that
\begin{equation}\label{r1x}
\|Ax(t)-b\|\leq\|A\|\|x(t)-x_t\|+\frac{c}{t^p}\|( {x}^*, {\lambda}^*)\|.
\end{equation}
This together with $(\ref{r1})$ yields
\begin{equation}\label{r2}
\|Ax(t)-b\|=\mathcal{O}\left( \sqrt{m(t)}t^{\frac{3q+s+p-2+r}{2}}+\frac{1}{t^p}\right),~\mbox{as}~t\to +\infty.
\end{equation}
By virtue of $(\ref{KKT})$ and $( {x}^*, {\lambda}^*)\in \Omega,$ we can easily get
\begin{equation}\label{r2xz}
\begin{split}
\mathcal{L}(x(t),\lambda^*)-\mathcal{L}(x^*,\lambda^*)
\leq&\mathcal{L}_t(x(t),\lambda_t)-\mathcal{L}_t(x_t,\lambda_t)
+\|\lambda_t-\lambda^*\| \|Ax(t)-b\|\\
&+\frac{c}{2t^p}(\|x^*\|^2-\|x(t)\|^2).
\end{split}
\end{equation}
Then, together with Lemma \ref{xtyt} (i), $(\ref{L2})$ and $(\ref{r2})$, it is easy to show that
$$
\mathcal{L}(x(t),\lambda^*)-\mathcal{L}(x^*,\lambda^*)=\mathcal{O}\left( \sqrt{m(t)}t^{\frac{3q+s+p-2+r}{2}}+\frac{1}{t^p}\right), ~~~\mbox{as}~t\to +\infty.
$$
Moreover, by $
|f(x(t))-f(x^*)|\leq\mathcal{L}(x(t),\lambda^*)-\mathcal{L}(x^*,\lambda^*)+\|\lambda^*\| \|Ax(t)-b\|,
$
we have
$$
|f(x(t))-f(x^*)|=\mathcal{O}\left( \sqrt{m(t)}t^{\frac{3q+s+p-2+r}{2}}+\frac{1}{t^p}\right),\mbox{as}~t\to +\infty.
$$

\textbf{Case II}: $ 0<p<\frac{1-q}{2} $.  In this case,  $0>2q+s-p-1> 3q+s+p-2$ . Consequently, from $(\ref{define})$,   there exists $\mathcal{C}''_0 \geq 0$ and  $ t'_{8}\geq t_7 $ such that for any $t\geq t'_8,$
\begin{equation*}
\frac{d}{dt}\left(e^{\frac{M}{1-r}t^{1-r}}\mathcal{E}(t) \right)\leq \mathcal{C}''_0 m(t)t^{2q+s-p-1}e^{\frac{M}{1-r}t^{1-r}}.
\end{equation*}
Then, using a similar argument as in Case I,  there exists  $\mathcal{C}'_3 \geq 0$ such that  for $t$ big enough,
\begin{equation*}
\mathcal{E}(t)\leq \mathcal{C}'_3 m(t)t^{2q+s-p-1+r}.
\end{equation*}
Then,  as $t\to +\infty$,  it holds
\begin{equation}\label{L2ax}
\mathcal{L}_t(x(t),\lambda_t)-\mathcal{L}_t(x_t,\lambda_t)= \mathcal{O}\left( t^{ -p-1+r}\right)
\end{equation}
and
\begin{equation}\label{r1a}
\|(x(t)-x_t,\lambda(t)-\lambda_t)\|^2= \mathcal{O}\left( m(t)t^{2q+s-p-1+r}\right).
\end{equation}
Together with $(\ref{r1x})$ and $(\ref{r1a})$, we have
\begin{equation}\label{r2q}
\|Ax(t)-b\|=\mathcal{O}\left( \sqrt{m(t)}t^{\frac{2q+s-p-1+r}{2}}+\frac{1}{t^p} \right),~\mbox{as}~t\to +\infty.
\end{equation}
Together with $(\ref{r2xz})$,  $(\ref{L2ax})$, and $(\ref{r2q})$, we obtain that
$$
\mathcal{L}(x(t),\lambda^*)-\mathcal{L}(x^*,\lambda^*)=\mathcal{O}\left( \sqrt{m(t)}t^{\frac{2q+s-p-1+r}{2}}+\frac{1}{t^p}\right),~\mbox{as}~t\to +\infty.
$$
Moreover,    as $t\to +\infty,$
we have
$$
|f(x(t))-f(x^*)|=\mathcal{O}\left( \sqrt{m(t)}t^{\frac{2q+s-p-1+r}{2}}+\frac{1}{t^p}\right).
$$

Finally,   from (\ref{r1}) and (\ref{r1a}), we have $\lim_{t\to +\infty} \|(x(t)-x_t,\lambda(t)-\lambda_t)\| =0$.
Then, by Lemma {\ref{xtyt}}, we have $\lim_{t\to +\infty} \|(x(t),\lambda(t))-( {x}^*, {\lambda}^*)\| =0$.  The proof is complete. \qed
\end{proof}
\begin{remark}
\begin{enumerate}
\item[{\rm (i)}]
In the special case that $m(t) \equiv 1$ and $\gamma\equiv1$, i.e.,  without  variable mass and Hessian-driven damping,  the dynamical system
\textup{(\ref{dyn1})}  reduces to the  system (\ref{zhu}) considered in \cite{Zhu2024S}. Thus, Theorem \ref{slow} extends the results obtained in
\cite{Zhu2024S}.
\item[{\rm (ii)}] Under the strong assumption   that $x(t)$ either  remains within the open ball $\mathbb{B}(0, \|x^*\|)$  or in its complement, strong convergence  results  in the inferior limit sense have also been obtained in  \cite{2026zhu,lihl,zhujcam,sunjota,Bot2021T,cs24k}. However, without relying on this assumption,  Theorem \ref{slow} provides a   strong convergence result  in the limit sense. Thus, Theorem \ref{slow} can also be regarded as a improvement of the results presented in \cite{2026zhu,lihl,zhujcam,sunjota,Bot2021T,cs24k}.
\end{enumerate}
\end{remark}

In the case that $ 0<p<\frac{1-q}{2}$, we can easily obtain  the following result   in terms of Theorem \ref{slow} (iii).
\begin{theorem}\label{slow100} Suppose that Assumptions \textup{\ref{assume1}}  and  \textup{\ref{assume2}} are satisfied. Let $ (x ,\lambda ):\left[t_0, +\infty\right)\to\mathcal{X} \times \mathcal{Y}  $ be a global solution of   system \textup{(\ref{dyn1})} and let $ (x_t,\lambda_t) $ be the saddle point of $ \mathcal{L}_t $. Suppose that  $ 0<p<\frac{1-q}{2}$, $5q+2s-1<0$, and $ 4q+s+p-2<0$. Then, for
$( {x}^*, {\lambda}^*)=\textup{Proj}_{\Omega}{0} $, we have
\begin{enumerate}
\item[{\rm (i)}] $(x(t),\lambda(t))$ converges strongly to $( {x}^*, {\lambda}^*)$, that is $
\lim_{t\to +\infty} \|(x(t),\lambda(t))-( {x}^*, {\lambda}^*)\| =0.
$

\item[{\rm (ii)}] If $ 0<p<2q+s$,   as $t\to +\infty$ it holds
\begin{eqnarray*}
\left\{\begin{split}
&\mathcal{L}(x(t),\lambda^*)-\mathcal{L}(x^*,\lambda^*)=\mathcal{O}\left(\sqrt{m(t)}t^{\frac{3q+s-p-1}{2}}+\frac{1}{t^p}\right), \\
&|f(x(t))-f(x^*)|=\mathcal{O}\left(\sqrt{m(t)}t^{\frac{3q+s-p-1}{2}}+\frac{1}{t^p}\right),\\
&\|Ax(t)-b\|=\mathcal{O}\left(\sqrt{m(t)}t^{\frac{3q+s-p-1}{2}}+\frac{1}{t^p}\right),\\
&\|(x(t)-x_t,\lambda(t)-\lambda_t)\|^2= \mathcal{O}\left(  {m(t)}t^{ {3q+s-p-1} }\right).
\end{split}\right.
\end{eqnarray*}

\item[{\rm (iii)}] If $ 2q+s\leq p<\frac{1-q}{2} $,  as $t\to +\infty$ it holds
\begin{eqnarray*}
\left\{\begin{split}
&\mathcal{L}(x(t),\lambda^*)-\mathcal{L}(x^*,\lambda^*)=\mathcal{O}\left(\sqrt{m(t)}t^{\frac{q-1}{2}}+\frac{1}{t^p}\right), \\
&|f(x(t))-f(x^*)|=\mathcal{O}\left(\sqrt{m(t)}t^{\frac{q-1}{2}}+\frac{1}{t^p}\right), \\
&\|Ax(t)-b\|=\mathcal{O}\left(\sqrt{m(t)}t^{\frac{q-1}{2}}+\frac{1}{t^p}\right), \\
&\|(x(t)-x_t,\lambda(t)-\lambda_t)\|^2= \mathcal{O}\left(  {m(t)}t^{q-1 }\right).
\end{split}\right.
\end{eqnarray*}
\end{enumerate}
\end{theorem}

\begin{proof} From $5q+2s-1<0$, we have $3q+s-p-1<0.$ Further, note that  $r=q$ when $ 0<p<2q+s$ and $r=p-q-s$ when  $ 2q+s\leq p<\frac{1-q}{2} $. Then, by vietue of Theorem \ref{slow} (iii), we can easily get the desired results. \qed
\end{proof}

In the case that $ \frac{1-q}{2}\leq p<1-q$, we can easily obtain  the following result   in terms of Theorem \ref{slow} (ii).

\begin{theorem}Suppose that Assumptions \textup{\ref{assume1}}  and  \textup{\ref{assume2}} are satisfied. Let $ (x ,\lambda ):\left[t_0, +\infty\right)\to\mathcal{X} \times \mathcal{Y}  $ be a global solution of   system \textup{(\ref{dyn1})} and let $ (x_t,\lambda_t) $ be the saddle point of $ \mathcal{L}_t $. Suppose that  $ \frac{1-q}{2}\leq p<1-q$, $\frac{1-q}{2}\leq2q+s<1-q$  and  $ 4q+s+p-2<0$ Then, for
$( {x}^*, {\lambda}^*)=\textup{Proj}_{\Omega}{0} $, we have
\begin{enumerate}
\item[{\rm (i)}] $(x(t),\lambda(t))$ converges strongly to $( {x}^*, {\lambda}^*)$, that is $
\lim_{t\to +\infty} \|(x(t),\lambda(t))-( {x}^*, {\lambda}^*)\| =0.
$

\item[{\rm (ii)}] If $ \frac{1-q}{2}<p<2q+s$, as $t\to +\infty$ it holds
\begin{eqnarray*}
\left\{\begin{split}
&\mathcal{L}(x(t),\lambda^*)-\mathcal{L}(x^*,\lambda^*)=\mathcal{O}\left(\sqrt{m(t)}t^{\frac{4q+s+p-2}{2}}+\frac{1}{t^p}\right), \\
&|f(x(t))-f(x^*)|=\mathcal{O}\left(\sqrt{m(t)}t^{\frac{4q+s+p-2}{2}}+\frac{1}{t^p}\right), \\
&\|Ax(t)-b\|=\mathcal{O}\left(\sqrt{m(t)}t^{\frac{4q+s+p-2}{2}}+\frac{1}{t^p}\right), \\
&\|(x(t)-x_t,\lambda(t)-\lambda_t)\|^2= \mathcal{O}\left( {m(t)}t^{ {4q+s+p-2} }\right).
\end{split}\right.
\end{eqnarray*}

\item[{\rm (iii)}] If $ 2q+s\leq p<1-q$, as $t\to +\infty$ it holds
\begin{eqnarray*}
\left\{\begin{split}
&\mathcal{L}(x(t),\lambda^*)-\mathcal{L}(x^*,\lambda^*)=\mathcal{O}\left(\sqrt{m(t)}t^{q+p-1}+\frac{1}{t^p}\right), \\
&|f(x(t))-f(x^*)|=\mathcal{O}\left(\sqrt{m(t)}t^{q+p-1}+\frac{1}{t^p}\right), \\
&\|Ax(t)-b\|=\mathcal{O}\left(\sqrt{m(t)}t^{q+p-1}+\frac{1}{t^p}\right), \\
&\|(x(t)-x_t,\lambda(t)-\lambda_t)\|^2= \mathcal{O}\left(  {m(t)}t^{2q+2p-2}\right).
\end{split}\right.
\end{eqnarray*}
\end{enumerate}
\end{theorem}

\begin{proof}
The proofs of items (ii) and (iii) are completed  based on Theorem \ref{slow} (ii)  and     are similar to the proof of Theorem \ref{slow100}. \qed
\end{proof}

\section{Numerical experiments}
In this section, we give some numerical experiments to demonstrate the obtained theoretical results. All codes are performed  on a PC (with 2.30GHz Intel Core i5-8300H and 8GB memory) under MATLAB Version R2018a.

\begin{example}\label{example5.1}
 Let $ Q = H^\top H \in \mathbb{R}^{n \times n} $, $  H \in \mathbb{R}^{n \times n}  $,  $ A \in \mathbb{R}^{m \times n} $, $k\in\mathbb{R}^{ n}$ and $b\in\mathbb{R}^{ m}$. All entries of $ H, A, k $ and $ b $ are generated by the standard Gaussian distribution. Consider the following quadratic optimization problem
\begin{eqnarray*}
\left\{ \begin{array}{ll}
&\mathop{\mbox{min}}\limits_{x\in\mathbb{R}^n}~~{\frac{1}{2}x^\top Q x}+k^\top x\\
&\mbox{s.t.}~~Ax=b.
\end{array}
\right.
\end{eqnarray*}
\end{example}

In this numerical experiment, we set  $m=5$, $n=10$. For the dynamical system (\ref{dyn1}), we consider the initial conditions $ (x(1), \lambda(1), \dot{x}(1)) = \boldsymbol{1}^{(2n+ m)\times 1}$, and take $\alpha=1.1$, $q=0.06$, $p=0.9$, $s=0.7$, $c=0.01$ and $\gamma=2$. In this setting on the parameters, we test numerical performance
of   $|f(x(t))-f(x^*)|$ and $\|Ax(t)-b\|$ under different choices of the mass function $m(t)\in\{1,\frac{1}{t^{0.1}},\frac{1}{t^{0.4}},\frac{1}{t^{0.7}}\}$.   The results are depicted in Figure \ref{f1}.
 \vspace{-2em}
 \begin{figure}[H]%
     \centering
    {
         \includegraphics[width=0.45\linewidth]{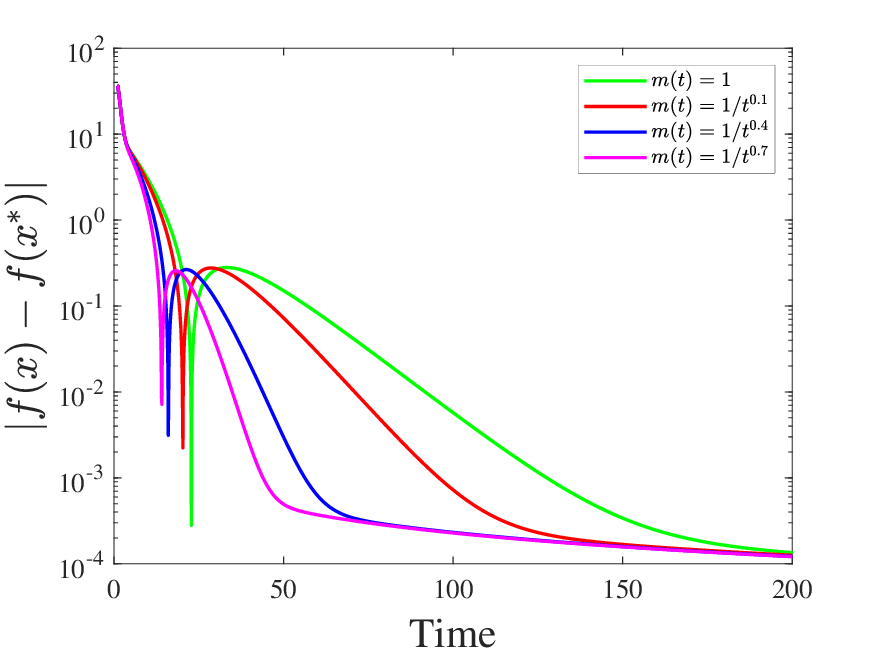}
         }\hfill
     {
         \includegraphics[width=0.45\linewidth]{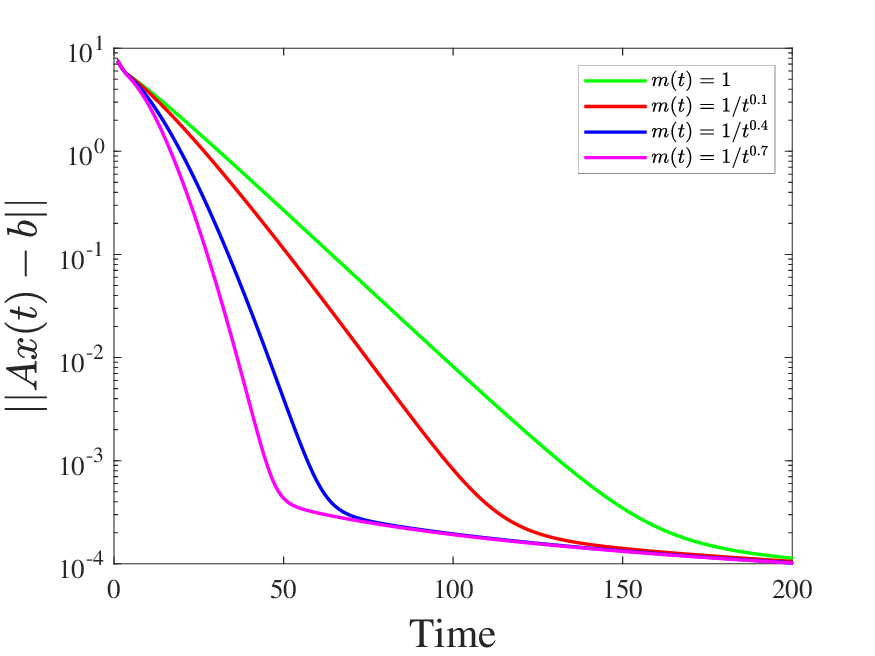}
         }

     \caption{Error analysis of   system (\ref{dyn1}) under different mass functions.}
     \label{f1}
 \end{figure}\vspace{-2em}

Figure \ref{f1}  shows that as the mass function $ m(t) $ takes smaller values, the dynamical system (\ref{dyn1}) performs better for the convergence of  the objective residual $|f(x(t))-f(x^*)|$ and constraint violation $\|Ax(t)-b\|$.

\begin{example}\label{example5.2} Let  $x=(x_1,x_2,x_3)^\top\in \mathbb{R}^3$, $f(x)=(mx_1+nx_2+ex_3)^2$, $A=(m,-n,e)$, $m,n,e\in\mathbb{R}\backslash \{0\}$ and $b=0$. Consider the following   convex optimization problem
\begin{eqnarray}\label{N1}
\left\{ \begin{array}{ll}
&\mathop{\mbox{min}}\limits_{x\in\mathbb{R}^3}~~{(mx_1+nx_2+ex_3)^2}\\
&\mbox{s.t.}~~mx_1-nx_2+ex_3=0.
\end{array}
\right.
\end{eqnarray}
\end{example}
Clearly, the solution set of this convex optimization problem (\ref{N1}) is $\{(x_1,0,-\frac{m}{e}x_1)^{\top}|x_1\in\mathbb{R}\}$, $x^*=(0,0,0)^\top$ is the minimal norm solution of problem (\ref{N1}),  and $f(x^*)=0$.

In the following numerical experiment, we investigate the influence of time scaling function $ t^s $. For problem (\ref{N1}), we consider $ m=1, n=2, e=1 $. Take the initial
condition  $x(1)=(1,1,-1)^\top$, $\lambda(1)=(1)$, $\dot{x}(1)=(-1,-1,1)^\top$. Set
  $\alpha = 3$, $q=0.1$, $p=0.1$, $c=5$, $\gamma=1$,  $m(t)=\frac{1}{t^{0.15}}$ and $s\in\{0.1, 0.3, 0.5, 0.7\}$. The behaviors of  $|f(x(t))-f(x^*)|$ and $\|Ax(t)-b\|$ of the dynamical
system (\ref{dyn1}) are depicted in Figure \ref{f2}.

\vspace{-2em}
 \begin{figure}[H]%
     \centering
     {
     \includegraphics[width=0.45\linewidth]{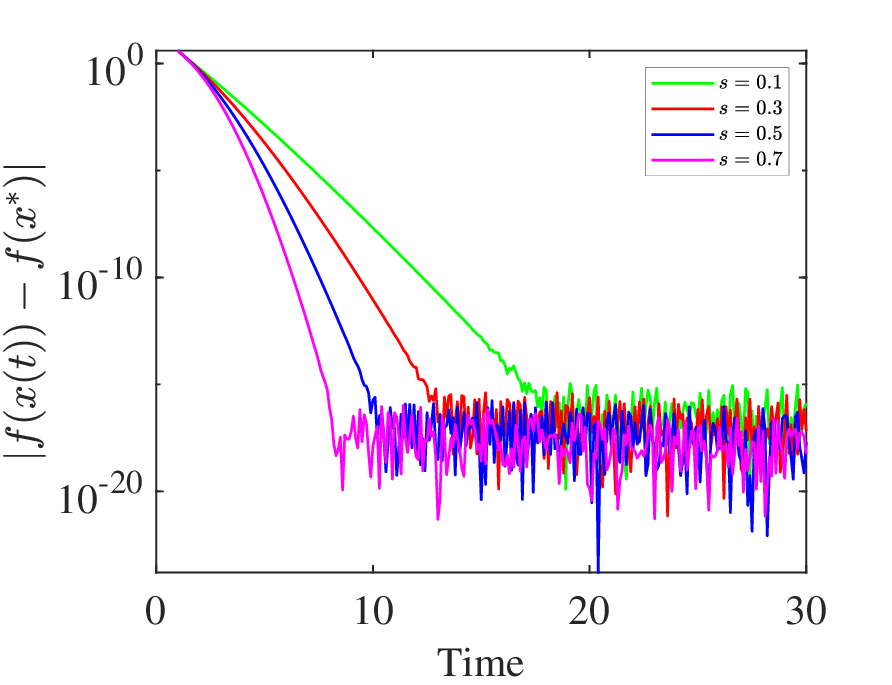}
         }\hfill
     {
         \includegraphics[width=0.45\linewidth]{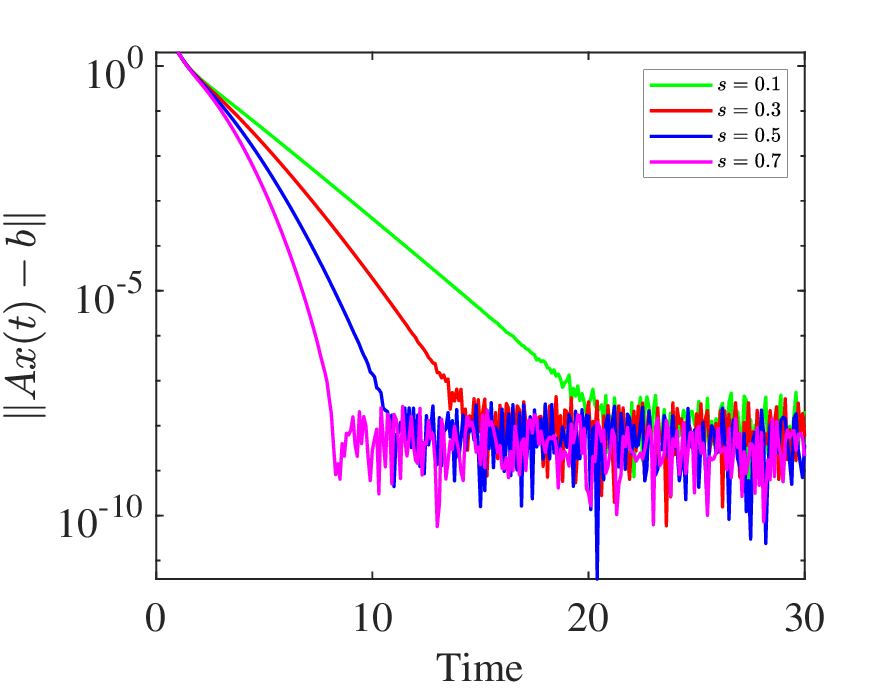}
         }

     \caption{Error analysis of   system (\ref{dyn1}) under different time scaling functions.}
     \label{f2}
 \end{figure}\vspace{-2em}

 As shown  in Figure  \ref{f2},  a faster-growing time scaling function helps to accelerate the  convergence  for both the objective residual $|f(x(t))-f(x^*)|$ and constraint violation $\|Ax(t)-b\|$.

Next, we  investigate the effect of Hessian-driven damping on the system.  For problem (\ref{N1}), we let $ m=10, n=20 , e=10 $. Take the same  initial
condition as above. Set
  $\alpha = 3$, $q=0.1$, $p=0.1$, $c=5$, $s=0.1$,  $m(t)=\frac{1}{t^{0.15}}$ and $\gamma\in\{0, 1\}$.  The behaviors of trajectory $x(t)$ generated by the dynamical
  system (\ref{dyn1}) are depicted in Figure \ref{f3}.

 \begin{figure}[htbp]
\begin{minipage}[t]{0.4\textwidth}
\centering
\includegraphics[scale=0.4]{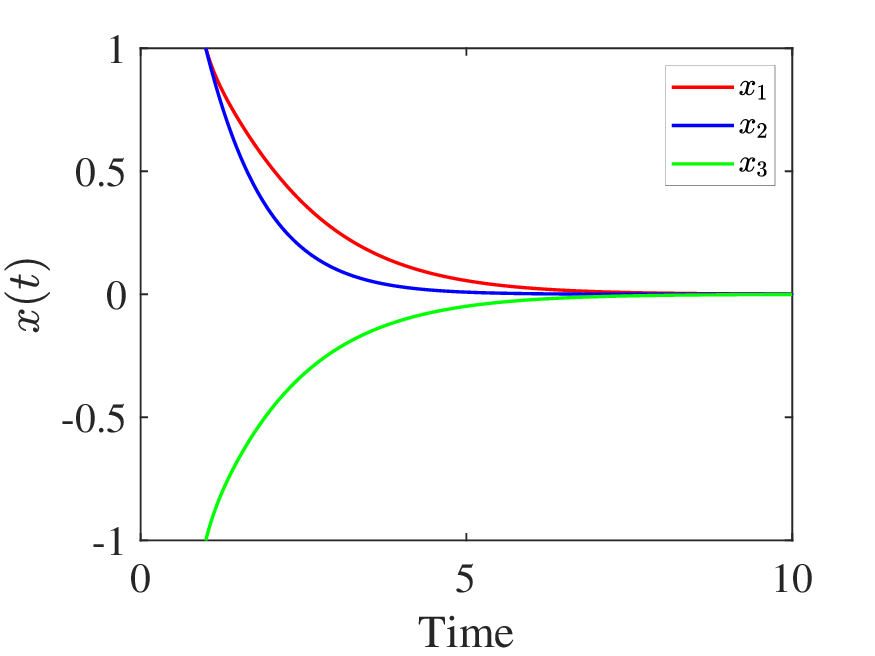}
\vspace{-3em}
\hspace{9pt}
\makebox[0.6\linewidth]{\fontsize{8.5}{8.5}\selectfont $({a})\gamma=1$}
\end{minipage}
\hfill
\begin{minipage}[t]{0.4\textwidth}
\centering
\includegraphics[scale=0.4]{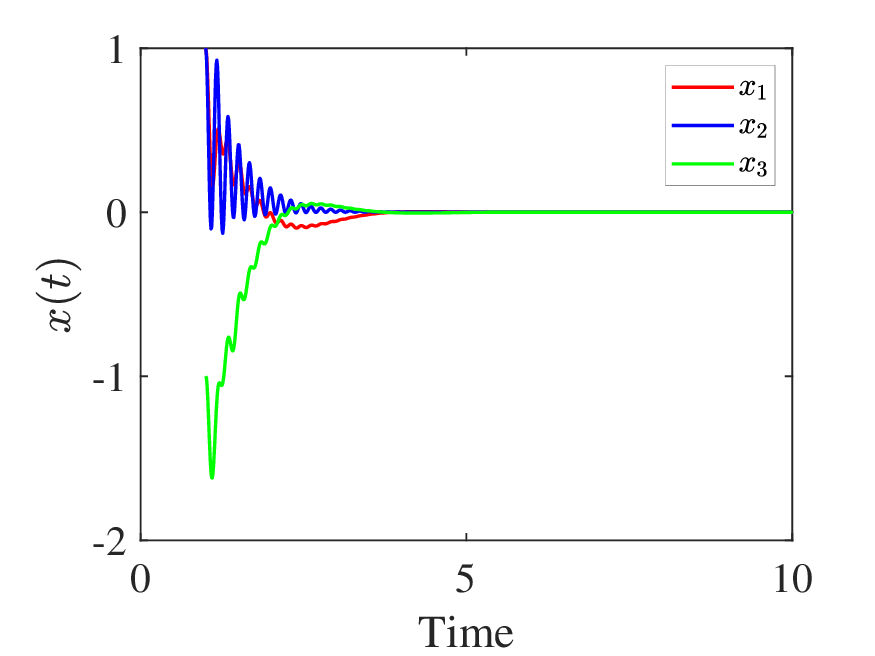}
\vspace{-3em}
\hspace{9pt}
\makebox[0.6\linewidth]{\fontsize{8.5}{8.5}\selectfont $({b})\gamma=0$}
\end{minipage}
\vspace{1em}
\caption{The behaviors of the trajectory  generated by  system (\ref{dyn1}).}
\vspace{-2em}
\label{f3}
\end{figure}
From Figure \ref{f3}, it can be observed that the Hessian-driven damping term enables the solution trajectory of dynamical system (\ref{dyn1}) to converge more smoothly to the optimal solution.

\section{Conclusion}
 In this paper, we introduce a Tikhonov regularized  primal-dual dynamical system (\ref{dyn1})  with variable mass and Hessian-driven damping to solve the convex optimization problem $(\ref{constrained})$. Compared with the dynamical systems introduced in \cite{Zhu2024S,He2023C} for solving the problem $(\ref{constrained})$, the system (\ref{dyn1}) not only incorporates slowly  viscous damping, extrapolation and time scaling, but is also governed by a constant Hessian-driven damping and variable mass. By employing  appropriate conditions on the underlying parameters, we establish the fast convergence rates of  the primal-dual gap, the objective residual, and the feasibility violation,   as well as  the strong asymptotic convergence of
the trajectory generated by the system (\ref{dyn1}).

Although some new results have been obtained on the   system (\ref{dyn1}) for solving the problem $(\ref{constrained})$, there are remaining questions to be
addressed in the future. For instance, an interesting   direction for research is  to investigate the explicit discretization  of the  system (\ref{dyn1}), which leads to  an inertial numerical algorithm   for solving the problem $(\ref{constrained})$. On the other hand, it is also important to
consider the   system (\ref{dyn1}) with  time-dependent Hessian-driven damping
in the future.

\section*{Funding}
\small{  This research is supported by the Natural Science Foundation of Chongqing (CSTB2024NSCQ-MSX0651) and the Team Building Project for Graduate Tutors in Chongqing (yds223010).}

\section*{Data availability}

 \small{ The authors confirm that all data generated or analysed during this study are included in this article.}

 \section*{Declaration}

 \small{\textbf{Conflict of interest} No potential conflict of interest was reported by the authors.}

\bibliographystyle{plain}

\end{document}